\documentclass[final,3p,times]{elsarticle}

\usepackage{graphicx}
\usepackage{epsfig}
\usepackage{amssymb}
\usepackage{amsmath}
\usepackage{amsthm}
\usepackage{multirow}
\usepackage{algorithm}
\usepackage{algorithmic}
\usepackage{array}

\journal{Applied and Computational Harmonic Analysis}

\theoremstyle{plain}
\newtheorem{theorem}{Theorem}
\newtheorem{lemma}[theorem]{Lemma}

\theoremstyle{definition}
\newtheorem{definition}[theorem]{Definition}

\theoremstyle{remark}

\newtheorem{example}[theorem]{Example}

\begin{document}

\begin{frontmatter}
\title{Two are better than one: Fundamental parameters of frame coherence}
\author[Rutgers]{Waheed U. Bajwa}
\author[Duke,Princeton]{Robert Calderbank}
\author[Princeton]{Dustin G. Mixon}

\address[Rutgers]{Department of Electrical and Computer Engineering, Rutgers, The State University of New Jersey, Piscataway, New Jersey 08854, USA}
\address[Duke]{Department of Electrical and Computer Engineering, Duke University, Durham, North Carolina 27708, USA}
\address[Princeton]{Program in Applied and Computational Mathematics, Princeton University, Princeton, New Jersey 08544, USA}

\begin{abstract}
This paper investigates two parameters that measure the coherence of a frame: worst-case and average coherence.
We first use worst-case and average coherence to derive near-optimal probabilistic guarantees on both sparse signal detection and reconstruction in the presence of noise.
Next, we provide a catalog of nearly tight frames with small worst-case and average coherence.
Later, we find a new lower bound on worst-case coherence; we compare it to the Welch bound and use it to interpret recently reported signal reconstruction results.
Finally, we give an algorithm that transforms frames in a way that decreases average coherence without changing the spectral norm or worst-case coherence.
\end{abstract}

\end{frontmatter}

\section{Introduction}

Many classical applications, such as radar and error-correcting codes, make use of over-complete spanning systems \cite{strohmer:acha03}.
Oftentimes, we may view an over-complete spanning system as a \emph{frame}.
Take $F=\{f_i\}_{i\in\mathcal{I}}$ to be a collection of vectors in some separable Hilbert space $\mathcal{H}$.
Then $F$ is a frame if there exist \emph{frame bounds} $A$ and $B$ with $0<A\leq B<\infty$ such that
$A\|x\|^2\leq\sum_{i\in\mathcal{I}}|\langle x,f_i\rangle|^2\leq B\|x\|^2$ for every $x\in\mathcal{H}$.
When $A=B$, $F$ is called a \emph{tight frame}.
For finite-dimensional unit norm frames, where $\mathcal{I}=\{1,\ldots,N\}$, the \emph{worst-case coherence} is a useful parameter:
\begin{equation}
\label{eq.mu defn}
\mu_F:=\max_{\substack{i,j\in\{1,\ldots,N\}\\i\neq j}}|\langle f_i,f_j\rangle|.
\end{equation}
Note that orthonormal bases are tight frames with $A=B=1$ and have zero worst-case coherence.
In both ways, frames form a natural generalization of orthonormal bases.

In this paper, we only consider finite-dimensional frames.
Those not familiar with frame theory can simply view a finite-dimensional frame as an $M\times N$ matrix of rank $M$ whose columns are the frame elements.
With this view, the tightness condition is equivalent to having the spectral norm be as small as possible; for an $M\times N$ unit norm frame $F$, this equivalently means $\|F\|_2^2=\frac{N}{M}$.

Throughout the literature, applications require finite-dimensional frames that are nearly tight and have small worst-case coherence \cite{candes:annstat09,donoho:tit06b,HP03,mixon:icassp11,strohmer:acha03,tropp:tit04,tropp:acha08,zahedi:acc10}.
Among these, a foremost application is sparse signal processing, where frames of small spectral norm and/or small worst-case coherence are commonly used to analyze sparse signals \cite{candes:annstat09,donoho:tit06b,tropp:tit04,tropp:acha08,zahedi:acc10}.
In general, sparse signal processing deals with measurements of the form
\begin{equation*}
y=Fx+e,
\end{equation*}
where $F$ is $M\times N$ with $M\ll N$, $x$ has at most $K$ nonzero entries, and $e$ is some sort of noise.
When given measurements $y$ of $x$, one might be asked to reconstruct the original sparse vector $x$, or to find the locations of its nonzero entries, or to simply determine whether $x$ is nonzero---each of these is a sparse signal processing problem.
In some applications, the signal $x$ is sparse in the identity basis, in which case $F$ represents the measurement process. 
In other applications, $x$ is sparse in an orthonormal basis or an overcomplete dictionary $G$ \cite{candes:acha10}. In this case, $F$ is a composition of $A$, the frame resulting from the measurement process, and $G$, the sparsifying dictionary, i.e., $F = AG$. 
We do not make a distinction between the two formulations in this paper, but our results are most readily interpretable in a physical setting for the former case.

Recently, \cite{bajwa:jcn10} introduced another notion of frame coherence called \emph{average coherence}:
\begin{equation}
\label{eq.nu defn}
\nu_F:=\tfrac{1}{N-1}\max_{i\in\{1,\ldots,N\}}\bigg|\sum_{\substack{j=1\\j\neq i}}^N\langle f_i,f_j\rangle\bigg|.
\end{equation}
Note that, in addition to having zero worst-case coherence, orthonormal bases also have zero average coherence.
Intuitively, worst-case coherence is a measure of dissimilarity between frame elements, whereas average coherence measures how well the frame elements are distributed in the unit hypersphere.
In sparse signal processing, there are a number of performance guarantees that depend only on worst-case coherence \cite{donoho:pnas03,elad:tit02,gribonval:tech02,tropp:tit04}.
These guarantees at best allow for sparsity levels on the order of $\sqrt{M}$.
Compressed sensing has brought guarantees that depend on the Restricted Isometry Property, which is much more difficult to check, but the guarantees allow for sparsity levels on the order of $\smash{\frac{M}{\log N}}$ \cite{baraniuk:ca08,candes:tit05,candes:tit06b}.
Recently, \cite{bajwa:jcn10} used worst-case and average coherence to produce \emph{probabilistic} guarantees that also allow for sparsity levels on the order of $\smash{\frac{M}{\log N}}$; these guarantees require that worst-case and average coherence together satisfy the following property:

\begin{definition}
We say an $M\times N$ unit norm frame $F$ satisfies the \emph{Strong Coherence Property} if
\begin{equation*}
\mbox{(SCP-1)}~~~\mu_F\leq\tfrac{1}{164\log N}\qquad\mbox{and}\qquad\mbox{(SCP-2)}~~~\nu_F\leq\tfrac{\mu_F}{\sqrt{M}},
\end{equation*}
where $\mu_F$ and $\nu_F$ are given by \eqref{eq.mu defn} and \eqref{eq.nu defn}, respectively.
\end{definition}

The reader should know that the constant $164$ is not particularly essential to the above definition; it is used in \cite{bajwa:jcn10} to simplify some analysis and make certain performance guarantees explicit, but the constant is by no means optimal.
This in mind, the requirement (SCP-1) can be interpreted more generally as $\mu_F=O(\tfrac{1}{\log N})$.
In the next section, we will use the Strong Coherence Property to continue the work of \cite{bajwa:jcn10}.
Where \cite{bajwa:jcn10} provided guarantees for noiseless reconstruction, we will produce near-optimal guarantees for signal detection and reconstruction from \emph{noisy} measurements of sparse signals.
These guarantees are related to those in \cite{candes:annstat09,donoho:tit06b,tropp:cras08,tropp:acha08}, and we will also elaborate on this relationship.

The results given in \cite{bajwa:jcn10} and Section~2, as well as the applications discussed in \cite{candes:annstat09,donoho:tit06b,HP03,mixon:icassp11,strohmer:acha03,tropp:tit04,tropp:acha08,zahedi:acc10}
demonstrate a pressing need for nearly tight frames with small worst-case and average coherence, especially in the area of sparse signal processing.
This paper offers three additional contributions in this regard.
In Section~3, we provide a sizable catalog of frames that exhibit small spectral norm, worst-case coherence, and average coherence.
With all three frame parameters provably small, these frames are guaranteed to perform well in relevant applications.
Next, performance in many applications is dictated by worst-case coherence \cite{candes:annstat09,donoho:tit06b,HP03,mixon:icassp11,strohmer:acha03,tropp:tit04,tropp:acha08,zahedi:acc10}.
It is therefore particularly important to understand which worst-case coherence values are achievable.
To this end, the Welch bound \cite{strohmer:acha03} is commonly used in the literature.
However, the Welch bound is only tight when the number of frame elements $N$ is less than the square of the spatial dimension $M$ \cite{strohmer:acha03}.
Another lower bound, given in \cite{MSEA03,XZG05}, beats the Welch bound when there are more frame elements, but it is known to be loose for real frames \cite{CHS96}.
Given this context, Section~4 gives a new lower bound on the worst-case coherence of real frames.
Our bound beats both the Welch bound and the bound in \cite{MSEA03,XZG05} when the number of frame elements far exceeds the spatial dimension.
Finally, since average coherence is so new, there is currently no intuition as to when (SCP-2) is satisfied.
In Section~5, we use ideas akin to the switching equivalence of graphs to transform a frame that satisfies (SCP-1) into another frame with the same spectral norm and worst-case coherence that additionally satisfies (SCP-2).

Throughout the paper, we make use of certain notations that we address here.
Recall, with big-O notation, that $f(n)=O(g(n))$ if there exists positive $C$ and $n_0$ such that for all $n>n_0$, $f(n)\leq Cg(n)$.
Also, $f(n)=\Omega(g(n))$ if $g(n)=O(f(n))$, and $f(n)=\Theta(g(n))$ if $f(n)=O(g(n))$ and $g(n)=O(f(n))$.
Additionally, we use $F_\mathcal{K}$ to denote the matrix whose columns are taken from the matrix $F$ according to the index set $\mathcal{K}$.
Similarly, we use $x_\mathcal{K}$ to denote the column vector whose entries are taken from the column vector $x$ according to the index set $\mathcal{K}$.
The column vector of the $T$ largest entries in column vector $x$ is denoted by $x_T$.
We also use $\|x\|$ to denote the $\ell^2$ norm of a vector $x$, while $\|F\|_2$ is the spectral norm of a matrix $F$. 
Lastly, we use a star ($*$) to denote the matrix adjoint, a dagger ($\dagger$) to denote the matrix pseudoinverse, 
and $\mathrm{I}_K$ to denote the $K\times K$ identity matrix.

\section{Worst-case and average coherence: Applications to sparse signal processing}

Frames with small spectral norm, worst-case coherence, and/or average coherence have found use in recent years with applications involving sparse signals.
Donoho et al.~used the worst-case coherence in \cite{donoho:tit06b} to provide uniform bounds on the signal and support recovery performance of combinatorial and convex optimization methods and greedy algorithms.
Later, Tropp \cite{tropp:acha08} and Cand\`{e}s and Plan \cite{candes:annstat09} used both the spectral norm and worst-case coherence to provide tighter bounds on the signal and support recovery performance of convex optimization methods for most support sets under the additional assumption that the sparse signals have independent nonzero entries with zero median.
Recently, Bajwa et al.~\cite{bajwa:jcn10} made use of the spectral norm and both coherence parameters to report tighter bounds on the noisy model selection and noiseless signal recovery performance of an incredibly fast greedy algorithm called \emph{one-step thresholding (OST)} for most support sets and \emph{arbitrary} nonzero entries.
In this section, we discuss further implications of the spectral norm and worst-case and average coherence of frames in applications involving sparse signals.

\subsection{The Weak Restricted Isometry Property}
A common task in signal processing applications is to test whether a collection of measurements corresponds to mere noise \cite{kay:98b}.
For applications involving sparse signals, one can test measurements $y \in \mathbb{C}^M$ against the null hypothsis $H_0: y = e$ and alternative hypothesis $H_1: y = Fx+e$, where the entries of the noise vector $e\in \mathbb{C}^M$ are independent, identical zero-mean complex-Gaussian random variables and the signal $x\in\mathbb{C}^N$ is $K$-sparse.
The performance of such signal detection problems is directly proportional to the energy in $Fx$ \cite{davenport:jstsp10,haupt:icassp07,kay:98b}.
In particular, existing literature on the detection of sparse signals \cite{davenport:jstsp10,haupt:icassp07} leverages the fact that $\|Fx\|^2 \approx \|x\|^2$ when $F$ satisfies the Restricted Isometry Property (RIP) of order $K$.
In contrast, we now show that the Strong Coherence Property also guarantees $\|Fx\|^2 \approx \|x\|^2$ for most $K$-sparse vectors.
We start with a definition:

\begin{definition}
\label{def:WRIP}
We say an $M\times N$ frame $F$ satisfies the \emph{$(K,\delta,p)$-Weak Restricted Isometry Property (Weak RIP)} if for every $K$-sparse vector $y \in \mathbb{C}^N$, a random permutation $x$ of $y$'s entries satisfies
\begin{equation}
\label{thmeqn:REP}
(1-\delta)\|x\|^2 \leq \|Fx\|^2 \leq (1+\delta)\|x\|^2
\end{equation}
with probability exceeding $1-p$.
\end{definition}

At first glance, it may seem odd that we introduce a random permutation when we might as well define Weak RIP in terms of a $K$-sparse vector whose support is drawn randomly from all $\smash{\binom{N}{K}}$ possible choices.
In fact, both versions would be equivalent in distribution, but we stress that in the present definition, the values of the nonzero entries of $x$ are \emph{not} random; rather, the only randomness we have is in the locations of the nonzero entries.
We wish to distinguish our results from those in \cite{candes:annstat09}, which explicitly require randomness in the values of the nonzero entries.
We also note the distinction between RIP and Weak RIP---Weak RIP requires that $F$ preserves the energy of \emph{most} sparse vectors.
Moreover, the manner in which we quantify ``most'' is important.
For each sparse vector, $F$ preserves the energy of most permutations of that vector, but for different sparse vectors, $F$ might not preserve the energy of permutations with the same support.
That is, unlike RIP, Weak RIP is \emph{not} a statement about the singular values of submatrices of $F$.
Certainly, matrices for which most submatrices are well-conditioned, such as those discussed in \cite{tropp:cras08,tropp:acha08}, will satisfy Weak RIP, but Weak RIP does not require this.
That said, the following theorem shows, in part, the significance of the Strong Coherence Property.

\begin{theorem}
\label{thm.WRIP}
Any $M\times N$ unit norm frame $F$ that satisfies the Strong Coherence Property
also satisfies the $(K,\delta,\frac{4K}{N^2})$-Weak Restricted Isometry Property
provided $N \geq 128$ and $\smash{2K\log{N} \leq \min\{\frac{\delta^2}{100\mu_F^2},M\}}$.
\end{theorem}
\begin{proof}
Let $x$ be as in Definition~\ref{def:WRIP}.
Note that \eqref{thmeqn:REP} is equivalent to $\smash{\big|\|Fx\|^2-\|x\|^2\big|\leq\delta\|x\|^2}$.
Defining $\mathcal{K}:=\{n:|x_n|>0\}$, then the Cauchy-Schwarz inequality gives
\begin{equation}
\label{pfeqn:REP_1}
\big|\|Fx\|^2-\|x\|^2\big|
=|x_\mathcal{K}^*(F_\mathcal{K}^*F_\mathcal{K}^{}-\mathrm{I}_K^{})x_\mathcal{K}^{}|
\leq\|x_\mathcal{K}^{}\|~\|(F_\mathcal{K}^*F_\mathcal{K}^{}-\mathrm{I}_K^{})x_\mathcal{K}^{}\|
\leq\sqrt{K}~\|x_\mathcal{K}^{}\|~\|(F_\mathcal{K}^*F_\mathcal{K}^{}-\mathrm{I}_K^{})x_\mathcal{K}^{}\|_\infty,
\end{equation}
where the last inequality uses the fact that $\|\cdot\|\leq\sqrt{K}~\|\cdot\|_\infty$ in $\mathbb{C}^K$.
We now consider \cite[Lemma~3]{bajwa:jcn10}, which states that for any $\epsilon \in [0,1)$ and $a \geq 1$, $\|(F_\mathcal{K}^*F_\mathcal{K}^{}-\mathrm{I}_K^{})x_\mathcal{K}^{}\|_\infty \leq \epsilon \|x_\mathcal{K}^{}\|$ with probability exceeding $\smash{1-4K\mathrm{e}^{- (\epsilon-\sqrt{K}\nu_F)^2/16(2+a^{-1})^2\mu_F^2}}$ provided $\smash{K \leq \min\{\epsilon^2\nu_F^{-2}, (1+a)^{-1}N\}}$. 
We claim that \eqref{pfeqn:REP_1} together with \cite[Lemma~3]{bajwa:jcn10} guarantee $\smash{\big|\|Fx\|^2-\|x\|^2\big|\leq\delta\|x\|^2}$ with probability exceeding $\smash{1-\frac{4K}{N^2}}$.
In order to establish this claim, we fix $\epsilon=10\mu\sqrt{2\log{N}}$ and $a=2\log{128}-1$. 
It is then easy to see that (SCP-1) gives $\epsilon < 1$, and also that (SCP-2) and $2K\log{N} \leq M$ give $K \leq \epsilon^2\nu_F^{-2}/9$. 
Therefore, since the assumption that $N \geq 128$ together with $2K\log{N} \leq M$ implies $K \leq (1+a)^{-1}N$, we obtain $\smash{\mathrm{e}^{- (\epsilon - \sqrt{K}\nu_F)^2/16(2+a^{-1})^2\mu_F^2} \leq \frac{1}{N^2}}$. 
The result now follows from the observation that $\smash{2K\log{N} \leq \frac{\delta^2}{100\mu_F^2}}$ implies $\sqrt{K}\epsilon \leq \delta$.
\end{proof}

This theorem shows that having small worst-case and average coherence is enough to guarantee Weak RIP.
This contrasts with related results by Tropp \cite{tropp:cras08,tropp:acha08} that require $F$ to be nearly tight.
In fact, the proof of Theorem~\ref{thm.WRIP} does not even use the full power of the Strong Coherence Property; instead of (SCP-1), it suffices to have $\smash{\mu_F\leq1/(15\!\sqrt{\log N})}$, part of what \cite{bajwa:jcn10} calls the Coherence Property.
Also, if $F$ has worst-case coherence $\smash{\mu_F=O(1/\!\sqrt{M})}$ and average coherence $\nu_F=O(1/M)$, then even if $F$ has large spectral norm, Theorem~\ref{thm.WRIP} states that $F$ preserves the energy of most $K$-sparse vectors with $K=O(M/\log N)$, i.e., the sparsity regime which is linear in the number of measurements.

\subsection{Reconstruction of sparse signals from noisy measurements}
Another common task in signal processing applications is to reconstruct a
$K$-sparse signal $x\in\mathbb{C}^N$ from a small collection of linear
measurements $y\in\mathbb{C}^M$. Recently, Tropp \cite{tropp:acha08} used both
the worst-case coherence and spectral norm of frames to find bounds on the
reconstruction performance of \emph{basis pursuit (BP)} \cite{donoho:siamjsc98}
for most support sets under the assumption that the nonzero entries of $x$ are
independent with zero median. In contrast, \cite{bajwa:jcn10} used the spectral
norm and worst-case and average coherence of frames to find bounds on the
reconstruction performance of OST for most support sets and \emph{arbitrary}
nonzero entries. However, both \cite{bajwa:jcn10} and \cite{tropp:acha08} limit
themselves to recovering $x$ in the absence of noise, corresponding to $y =
Fx$, a rather ideal scenario.

Our goal in this section is to provide guarantees for the reconstruction of
sparse signals from noisy measurements $y=Fx+e$, where the entries of the noise
vector $e\in \mathbb{C}^M$ are independent, identical complex-Gaussian random
variables with mean zero and variance $\sigma^2$. In particular, and in
contrast with \cite{donoho:tit06b}, our guarantees will hold for arbitrary unit norm
frames $F$ without requiring the signal's sparsity level to satisfy $K=O(\mu_F^{-1})$. 
The reconstruction algorithm that we analyze here is the OST
algorithm of \cite{bajwa:jcn10}, which is described in
Algorithm~\ref{alg:OST_recon}. The following theorem extends the analysis of
\cite{bajwa:jcn10} and shows that the OST algorithm leads to near-optimal
reconstruction error for certain important classes of sparse signals.

Before proceeding further, we first define some notation. We use
$\textsf{\textsc{snr}}:=\|x\|^2/\mathbb{E}[\|e\|^2]$ to denote the
\emph{signal-to-noise ratio} associated with the signal reconstruction problem.
Also, we use $\smash{\mathcal{T}_\sigma(t):=\{n: |x_n| >
\frac{2\sqrt{2}}{1-t}\sqrt{2 \sigma^2 \log{N}}\}}$ for any $t \in (0,1)$ to
denote the locations of all the entries of $x$ that, roughly speaking, lie
above the \emph{noise floor} $\sigma$. Finally, we use
$\smash{\mathcal{T}_\mu(t):=\{n: |x_n| > \frac{20}{t}\mu_F\|x\|\sqrt{2
\log{N}}\}}$ to denote the locations of entries of $x$ that, roughly speaking,
lie above the \emph{self-interference floor} $\mu_F\|x\|$.
\begin{algorithm*}[t]
\caption{One-Step Thresholding (OST) for sparse signal reconstruction \cite{bajwa:jcn10}}
\label{alg:OST_recon}
\textbf{Input:} An $M \times N$ unit norm frame $F$, a vector $y=Fx+e$, and a threshold $\lambda > 0$\\
\textbf{Output:} An estimate $\hat{x} \in \mathbb{C}^N$ of the true sparse signal $x$
\begin{algorithmic}
\STATE $\hat{x} \leftarrow 0$ \hfill \COMMENT{Initialize}
\STATE $z \leftarrow F^* y$ \hfill \COMMENT{Form signal proxy}
\STATE $\hat{\mathcal{K}} \leftarrow \{n : |z_n| > \lambda\}$ \hfill \COMMENT{Select indices via OST}
\STATE $\hat{x}_{\hat{\mathcal{K}}} \leftarrow (F_{\hat{\mathcal{K}}})^\dagger y$ \hfill \COMMENT{Reconstruct signal via least-squares}
\end{algorithmic}
\end{algorithm*}

\begin{theorem}[Reconstruction of sparse signals]
\label{thm:RSP}
Take an $M\times N$ unit norm frame $F$ which satisfies the Strong Coherence Property, pick $t\in(0,1)$, and choose $\smash{\lambda = \sqrt{2\sigma^2\log{N}}~\max \{\frac{10}{t}\mu_F\sqrt{M~\textsf{\textsc{snr}}}, \frac{\sqrt{2}}{1-t}\}}$. 
Further, suppose $x \in \mathbb{C}^N$ has support $\mathcal{K}$ drawn uniformly at random from all possible $K$-subsets of $\{1,\ldots,N\}$.
Then provided
\begin{equation}
\label{thmeqn:RSP}
K \leq \tfrac{N}{c_1^2\|F\|_2^2\log{N}},
\end{equation}
Algorithm~\ref{alg:OST_recon} produces $\hat{\mathcal{K}}$ such that $\mathcal{T}_\sigma(t) \cap \mathcal{T}_\mu(t) \subseteq \hat{\mathcal{K}} \subseteq \mathcal{K}$ and $\hat{x}$ such that
\begin{equation}
\label{thmeqn:RSP_2}
\|x-\hat{x}\| \leq c_2 \sqrt{\sigma^2|\hat{\mathcal{K}}|\log{N}} + c_3\|x_{\mathcal{K} \setminus \hat{\mathcal{K}}}\|
\end{equation}
with probability exceeding $1 - 10N^{-1}$. Finally, defining $T:=|\mathcal{T}_\sigma(t) \cap \mathcal{T}_\mu(t)|$, we further have
\begin{equation}
\label{thmeqn:RSP_3}
\|x-\hat{x}\| \leq c_2 \sqrt{\sigma^2 K \log{N}} + c_3\|x - x_T\|
\end{equation}
in the same probability event.
Here, $c_1 = 37\mathrm{e}$, $c_2 = \frac{2}{1-\mathrm{e}^{-1/2}}$, and $c_3 = 1 + \frac{\mathrm{e}^{-1/2}}{1-\mathrm{e}^{-1/2}}$ are numerical constants.
\end{theorem}

\begin{proof}
To begin, note that since $\smash{\|F\|_2^2\geq\frac{N}{M}}$, we have from \eqref{thmeqn:RSP} that $K\leq M/(2\log{N})$. 
It is then easy to conclude from \cite[Theorem~5]{bajwa:jcn10} that $\smash{\hat{\mathcal{K}}}$ satisfies $\mathcal{T}_\sigma(t) \cap \mathcal{T}_\mu(t) \subseteq \hat{\mathcal{K}} \subseteq \mathcal{K}$ with probability exceeding $1 - 6N^{-1}$. Therefore, conditioned on the event $\smash{\mathcal{E}_1 := \{\mathcal{T}_\sigma(t) \cap \mathcal{T}_\mu(t) \subseteq \hat{\mathcal{K}} \subseteq \mathcal{K}\}}$, we can make use of the triangle inequality to write
\begin{equation}
\label{pfeqn:RSP_1}
\|x - \hat{x}\|\leq \|x_{\hat{\mathcal{K}}} - \hat{x}_{\hat{\mathcal{K}}}\| + \|x_{\mathcal{K}\setminus\hat{\mathcal{K}}}\|.
\end{equation}
Next, we may use \eqref{thmeqn:RSP} and the fact that $F$ satisfies the Strong Coherence Property to conclude from 
\cite{tropp:cras08} (see, e.g., \cite[Proposition~3]{bajwa:jcn10}) that 
$\|F_\mathcal{K}^*F_\mathcal{K}^{}-\mathrm{I}_K^{}\|_2 < \mathrm{e}^{-1/2}$ with probability exceeding $1 - 2N^{-1}$. 
Hence, conditioning on $\mathcal{E}_1$ and $\smash{\mathcal{E}_2 := \{\|F_\mathcal{K}^*F_\mathcal{K}^{}-\mathrm{I}_K^{}\|_2 < \mathrm{e}^{-1/2}\}}$, we have that $\smash{(F_{\hat{\mathcal{K}}})^\dagger = (F_{\hat{\mathcal{K}}}^* F_{\hat{\mathcal{K}}}^{})^{-1} F_{\hat{\mathcal{K}}}^*}$ since $F_{\hat{\mathcal{K}}}$ is a submatrix of a full column rank matrix $F_\mathcal{K}$.
Therefore, given $\mathcal{E}_1$ and $\mathcal{E}_2$, we may write
\begin{equation}
\label{pfeqn:RSP_1.5}
\hat{x}_{\hat{\mathcal{K}}} 
= (F_{\hat{\mathcal{K}}})^\dagger (Fx+e) 
=  x_{\hat{\mathcal{K}}} + (F_{\hat{\mathcal{K}}})^\dagger F_{\mathcal{K} \setminus \hat{\mathcal{K}}}x_{\mathcal{K} \setminus \hat{\mathcal{K}}} + (F_{\hat{\mathcal{K}}})^\dagger e,
\end{equation}
and so substituting \eqref{pfeqn:RSP_1.5} into \eqref{pfeqn:RSP_1} and applying the triangle inequality gives
\begin{align}
\nonumber
\|x - \hat{x}\| 
&\leq \|(F_{\hat{\mathcal{K}}})^\dagger F_{\mathcal{K} \setminus \hat{\mathcal{K}}}x_{\mathcal{K} \setminus \hat{\mathcal{K}}}\| + \|(F_{\hat{\mathcal{K}}})^\dagger e\| + \|x_{\mathcal{K} \setminus \hat{\mathcal{K}}}\|\\
\label{pfeqn:RSP_2}
&\leq \Big(1 + \|(F_{\hat{\mathcal{K}}}^* F_{\hat{\mathcal{K}}}^{})^{-1}\|_2 \|F_{\hat{\mathcal{K}}}^* F_{\mathcal{K} \setminus \hat{\mathcal{K}}}^{}\|_2\Big)\|x_{\mathcal{K} \setminus \hat{\mathcal{K}}}^{}\| + \|(F_{\hat{\mathcal{K}}}^* F_{\hat{\mathcal{K}}}^{})^{-1}\|_2 \|F_{\hat{\mathcal{K}}}^* e\|.
\end{align}
Since, given $\mathcal{E}_1$, we have that $\smash{F_{\hat{\mathcal{K}}}^* F_{\hat{\mathcal{K}}}^{} - \mathrm{I}_K^{}}$ and $\smash{F_{\hat{\mathcal{K}}}^* F_{\mathcal{K} \setminus \hat{\mathcal{K}}}^{}}$ are submatrices of $\smash{F_\mathcal{K}^* F_\mathcal{K}^{} - \mathrm{I}_K^{}}$, and since the spectral norm of a matrix provides an upper bound for the spectral norms of its submatrices, we have the following given $\mathcal{E}_1$ and $\mathcal{E}_2$:
$\smash{\|F_{\hat{\mathcal{K}}}^* F_{\mathcal{K} \setminus \hat{\mathcal{K}}}^{}\|_2 \leq \mathrm{e}^{-1/2}}$
and
$\smash{\|(F_{\hat{\mathcal{K}}}^* F_{\hat{\mathcal{K}}}^{})^{-1}\|_2 \leq \tfrac{1}{1-\mathrm{e}^{-1/2}}}$.
We can now substitute these bounds into \eqref{pfeqn:RSP_2} and make use of the fact that $\smash{\|F_{\hat{\mathcal{K}}}^* e\| \leq |\hat{\mathcal{K}}|^{1/2}\|F_{\hat{\mathcal{K}}}^* e\|_\infty}$ to conclude that
\begin{equation*}
\|x - \hat{x}\| \leq \tfrac{|\hat{\mathcal{K}}|^{1/2}}{1-\mathrm{e}^{-1/2}} \|F_{\hat{\mathcal{K}}}^* e\|_\infty + \Big(1 + \tfrac{\mathrm{e}^{-1/2}}{1-\mathrm{e}^{-1/2}}\Big)\|x_{\mathcal{K} \setminus \hat{\mathcal{K}}}\|,
\end{equation*}
given $\mathcal{E}_1$ and $\mathcal{E}_2$. 
At this point, define the event $\smash{\mathcal{E}_3 = \{\|F_{\hat{\mathcal{K}}}^* e\|_\infty \leq 2\sqrt{\sigma^2 \log{N}}\}}$ and note from \cite[Lemma~6]{bajwa:jcn10} that $\smash{\Pr(\mathcal{E}_3^\mathrm{c}) \leq 2(\sqrt{2\pi\log{N}}~N)^{-1}}$. 
A union bound therefore gives \eqref{thmeqn:RSP_2} with probability exceeding $1 - 10N^{-1}$. 
For \eqref{thmeqn:RSP_3}, note that $\hat{\mathcal{K}} \subseteq \mathcal{K}$ implies $|\hat{\mathcal{K}}| \leq K$, and so $\mathcal{T}_\sigma(t) \cap \mathcal{T}_\mu(t) \subseteq \hat{\mathcal{K}}$ implies that $\|x_{\mathcal{K} \setminus \hat{\mathcal{K}}}\| \leq \|x_{\mathcal{K} \setminus (\mathcal{T}_\sigma(t) \cap \mathcal{T}_\mu(t))}\| = \|x - x_T\|$.
\end{proof}

A few remarks are in order now for Theorem~\ref{thm:RSP}. First, if $F$
satisfies the Strong Coherence Property \emph{and} $F$ is nearly tight, then
OST handles sparsity that is almost linear in $M$: $K = O(M/\log{N})$ from
\eqref{thmeqn:RSP}. 
Second, we do not impose any control over the size of $T$, but rather we state the result in generality in terms of $T$; its size is determined by the signal class $x$ belongs to, the worst-case coherence of the frame $F$ we use to measure $x$, and the magnitude of the noise that perturbs $Fx$.
Third, the $\ell_2$ error associated with the OST
algorithm is the near-optimal (modulo the $\log$ factor) error of
$\smash{\sqrt{\sigma^2 K \log{N}}}$ \emph{plus} the best $T$-term approximation
error caused by the inability of the OST algorithm to recover signal entries
that are smaller than $\smash{O(\mu_F\|x\|\sqrt{2 \log{N}})}$. 
In particular, if the $K$-sparse signal $x$, the worst-case coherence $\mu_F$, and the noise $e$ together satisfy $\|x - x_T\| = O(\smash{\sqrt{\sigma^2 K \log{N}}})$,
then the OST algorithm succeeds with a near-optimal $\ell_2$ error of
$\smash{\|x-\hat{x}\| = O(\sqrt{\sigma^2 K \log{N}})}$. 
To see why this error is near-optimal, note that a $K$-dimension vector of random entries with mean zero and variance $\sigma^2$ has expected squared norm $\sigma^2 K$; in our case, we pay an additional log factor to find the locations of the $K$ nonzero entries among the entire $N$-dimensional signal.
It is important to recognize that the optimality condition $\|x - x_T\| = O(\smash{\sqrt{\sigma^2 K \log{N}}})$
depends on the signal class, the noise variance, and the worst-case coherence of the frame; in particular, the condition is satisfied whenever $\|x_{\mathcal{K} \setminus
\mathcal{T}_\mu(t)}\| = O(\smash{\sqrt{\sigma^2 K \log{N}}})$, since
\begin{equation*}
    \|x - x_T\| \leq \|x_{\mathcal{K} \setminus \mathcal{T}_\sigma(t)}\| +
\|x_{\mathcal{K} \setminus \mathcal{T}_\mu(t)}\| = O\Big(\sqrt{\sigma^2
K \log{N}}\Big) + \|x_{\mathcal{K} \setminus \mathcal{T}_\mu(t)}\|.
\end{equation*}
The following lemma provides classes of sparse signals that satisfy
$\|x_{\mathcal{K} \setminus \mathcal{T}_\mu(t)}\| =
O(\smash{\sqrt{\sigma^2 K \log{N}}})$
given sufficiently small noise variance and worst-case coherence, and consequently the OST
algorithm is near-optimal for the reconstruction of such signal classes.
\begin{lemma}\label{lem:OST_opt_cond}
Take an $M \times N$ unit norm frame $F$ with worst-case coherence
$\smash{\mu_F\leq\frac{c_0}{\sqrt{M}}}$ for some $c_0>0$, and suppose that $\smash{K\leq\frac{N}{c_1^2\|F\|_2^2\log N}}$ for some $c_1>0$.
Fix a constant $\beta \in (0,1]$, and suppose the magnitudes of $\beta K$ nonzero entries of $x$ are some $\alpha =
\Omega(\sqrt{\sigma^2 \log{N}})$, while the magnitudes of the remaining
$(1-\beta)K$ nonzero entries are not necessarily same, but are smaller than $\alpha$ and scale as $\smash{O(\sqrt{\sigma^2 \log{N}})}$. 
Then $\smash{\|x_{\mathcal{K} \setminus \mathcal{T}_\mu(t)}\| = O(\sqrt{\sigma^2 K \log{N}})}$, provided $\smash{c_0\leq\frac{tc_1}{20\sqrt{2}}}$.
\end{lemma}
\begin{proof}
Let $\mathcal{K}$ be the support of $x$, and define $\mathcal{I} := \{n : |x_n| = \alpha\}$.
We wish to show that $\mathcal{I}\subseteq\mathcal{T}_\mu(t)$, since this implies $\smash{\|x_{\mathcal{K} \setminus
\mathcal{T}_\mu(t)}\| \leq \|x_{\mathcal{K} \setminus \mathcal{I}}\| =
O(\sqrt{\sigma^2 K\log{N}})}$.
In order to prove $\mathcal{I}\subseteq\mathcal{T}_\mu(t)$, notice that
\begin{equation*}
\|x\|^2 
= \|x_{\mathcal{I}}\|^2+\|x_{\mathcal{K}\setminus\mathcal{I}}\|^2 
< \beta K \alpha^2 + (1-\beta)K\alpha^2 
= K\alpha^2,
\end{equation*}
and so combining this with the fact that $\|F\|_2^2\geq\frac{N}{M}$ gives
\begin{equation*}
\mu_F \|x\| \sqrt{\log{N}}
< \tfrac{c_0}{\sqrt{M}} \sqrt{K} \alpha \sqrt{\log{N}}
\leq \tfrac{c_0}{\sqrt{M}} \sqrt{\tfrac{N}{c_1^2\|F\|_2^2\log N}} ~\alpha\sqrt{\log{N}}
\leq \tfrac{c_0}{c_1}\alpha.
\end{equation*}
Therefore, provided $\smash{c_0\leq\frac{tc_1}{20\sqrt{2}}}$, we have that $\mathcal{I}\subseteq\mathcal{T}_\mu(t)$.
\end{proof}
In words, Lemma~\ref{lem:OST_opt_cond} implies that OST is near-optimal for
those $K$-sparse signals whose entries above the noise floor have roughly the
same magnitude. This subsumes a very important class of signals
that appears in applications such as multi-label prediction
\cite{HsuNips2009}, in which all the nonzero entries take values $\pm \alpha$.
To the best of our knowledge, Theorem~\ref{thm:RSP} is the first result in the sparse signal
processing literature that does not require RIP and still provides near-optimal
reconstruction guarantees for such signals from noisy measurements, while
using either random or deterministic frames, even when $K = O(M/\log{N})$.

We note that our techniques can be extended to reconstruct noisy signals, that is, we may consider measurements of the form $y=F(x+n)+e$, where $n\in\mathbb{C}^N$ is also a noise vector of independent, identical zero-mean complex-Gaussian random variables.
In particular, if the frame $F$ is tight, then our measurements will not color the noise, and so noise in the signal may be viewed as noise in the measurements: $y=Fx+(Fn+e)$; if the frame is not tight, then the noise will become correlated in the measurements, and performance would be depend nontrivially on the frame's Gram matrix.
Also, the authors have had some success with generalizing Theorem~\ref{thm:RSP} to approximately sparse signals; the analysis follows similiar lines, but is rather cumbersome, and it appears as though the end result is only strong enough in the case of very nearly sparse signals.
As such, we omit this result.

\section{Frame constructions}

In this section, we consider a range of nearly tight frames with small worst-case and average coherence.  
We investigate various ways of selecting frames at random from different libraries, and we show that for each of these frames, the spectral norm, worst-case coherence, and average coherence are all small with high probability.
Later, we will consider deterministic constructions that use Gabor and chirp systems, spherical designs, equiangular tight frames, and error-correcting codes.
For the reader's convenience, all of these constructions are summarized in Table~\ref{table.constructions}.
Before we go any further, recall the following lower bound on worst-case coherence:

\begin{theorem}[Welch bound \cite{strohmer:acha03}]
\label{thm.welch bound}
Every $M\times N$ unit norm frame $F$ has worst-case coherence $\mu_F\geq\sqrt{\tfrac{N-M}{M(N-1)}}$.
\end{theorem}

We will use the Welch bound in the proof of the following lemma, which gives three different sufficient conditions for a frame to satisfy (SCP-2).
These conditions will prove quite useful in this section and throughout the paper.

\begin{lemma} 
\label{lem.sufficient conditions}
For any $M\times N$ unit norm frame $F$, each of the following conditions implies $\nu_F\leq\frac{\mu_F}{\sqrt{M}}$:
\begin{enumerate}
\item[(i)] $\langle f_k,\sum_{n=1}^N f_n\rangle=\frac{N}{M}$ for every $k=1,\ldots,N$,
\item[(ii)] $N\geq2M$ and $\sum_{n=1}^N f_n=0$,
\item[(iii)] $N\geq M^2+3M+3$ and $\|\sum_{n=1}^N f_n\|^2\leq N$.
\end{enumerate}
\end{lemma}

\begin{proof}
For condition (i), we have
\begin{equation*}
\nu_F
=\tfrac{1}{N-1}\max_i\bigg|\sum_{\substack{j=1\\j\neq i}}^N\langle f_i,f_j\rangle\bigg|
=\tfrac{1}{N-1}\max_i\bigg|\bigg\langle f_i,\sum_{j=1}^N f_j\bigg\rangle-1\bigg|
=\tfrac{1}{N-1}\big(\tfrac{N}{M}-1\big).
\end{equation*}
The Welch bound therefore gives
$\nu_F=\tfrac{1}{N-1}\big(\tfrac{N}{M}-1\big)=\tfrac{N-M}{M(N-1)}\leq\mu_F\sqrt{\tfrac{N-M}{M(N-1)}}
\leq\tfrac{\mu_F}{\sqrt{M}}$.
For condition (ii), we have
\begin{equation*}
\nu_F
=\tfrac{1}{N-1}\max_i\bigg|\sum_{\substack{j=1\\j\neq i}}^N\langle f_i,f_j\rangle\bigg|
=\tfrac{1}{N-1}\max_i\bigg|\bigg\langle f_i,\sum_{j=1}^N f_j\bigg\rangle-1\bigg|
=\tfrac{1}{N-1}.
\end{equation*}
Considering the Welch bound, it suffices to show $\frac{1}{N-1}\leq\frac{1}{\sqrt{M}}\sqrt{\frac{N-M}{M(N-1)}}$.
Rearranging equivalently gives
\begin{equation}
\label{eq.lemma ii.2}
N^2-(M+1)N-M(M-1)\geq0.
\end{equation}
When $N=2M$, the left-hand side of \eqref{eq.lemma ii.2} becomes $(M-1)^2$, which is trivially nonnegative.  Otherwise, we have
\begin{equation*}
N\geq2M+1\geq M+1+\sqrt{M(M-1)}\geq\tfrac{M+1}{2}+\sqrt{\big(\tfrac{M+1}{2}\big)^2+M(M-1)}.
\end{equation*}
In this case, by the quadratic formula and the fact that the left-hand side of \eqref{eq.lemma ii.2} is concave up in $N$, we have that \eqref{eq.lemma ii.2} is indeed satisfied.
For condition (iii), we use the triangle and Cauchy-Schwarz inequalities to get
\begin{equation*}
\nu_F
=\tfrac{1}{N-1}\max_i\bigg|\bigg\langle f_i,\sum_{j=1}^N f_j\bigg\rangle-1\bigg|
\leq\tfrac{1}{N-1}\bigg(\max_i\bigg|\bigg\langle f_i,\sum_{j=1}^N f_j\bigg\rangle\bigg|+1\bigg)
\leq\tfrac{\sqrt{N}+1}{N-1}.
\end{equation*}
Considering the Welch bound, it suffices to show $\smash{\frac{\sqrt{N}+1}{N-1}\leq\frac{1}{\sqrt{M}}\sqrt{\frac{N-M}{M(N-1)}}}$.
Taking $x:=\sqrt{N}$ and rearranging gives a polynomial: 
$x^4-(M^2+M+1)x^2-2M^2x-M(M-1)\geq0$.
By convexity and monotonicity of the polynomial in $[M+\frac{3}{2},\infty)$, it can be shown that the largest real root of this polynomial is always smaller than $\smash{M+\frac{3}{2}}$.
Also, considering it is concave up in $x$, it suffices that $\smash{\sqrt{N}=x\geq M+\frac{3}{2}}$, which we have since
$N\geq M^2+3M+3\geq(M+\frac{3}{2})^2$.
\end{proof}

\subsection{Normalized Gaussian frames}

Construct a matrix with independent, Gaussian-distributed entries that have zero mean and unit variance.
By normalizing the columns, we get a matrix called a \emph{normalized Gaussian frame}.
This is perhaps the most widely studied type of frame in the signal processing and statistics literature.
To be clear, the term ``normalized'' is intended to distinguish the results presented here from results reported in earlier works, such as \cite{bajwa:jcn10,baraniuk:ca08,candes:tit05,wainwright:tit09}, 
which only ensure that Gaussian frame elements have unit norm in expectation. 
In other words, normalized Gaussian frame elements are independently and uniformly distributed on the unit hypersphere in $\mathbb{R}^M$.
That said, the following theorem characterizes the spectral norm and the worst-case and average coherence of normalized Gaussian frames.

\begin{theorem}[Geometry of normalized Gaussian frames]
\label{thm.normalized gaussian frames}
Build a real $M\times N$ frame $G$ by drawing entries independently at random from a Gaussian distribution of zero mean and unit variance.
Next, construct a normalized Gaussian frame $F$ by taking $\smash{f_n:=\frac{g_n}{\|g_n\|}}$ for every $n=1,\ldots,N$.
Provided $\smash{60\log{N}\leq M\leq\frac{N-1}{4\log{N}}}$, then the following inequalities simultaneously hold with probability exceeding $1 - 11N^{-1}$:
\begin{enumerate}
\item[(i)] $\mu_F \leq \frac{\sqrt{15\log{N}}}{\sqrt{M} - \sqrt{12\log{N}}}$,
\item[(ii)] $\nu_F \leq \frac{\sqrt{15\log{N}}}{M - \sqrt{12M\log{N}}}$,
\item[(iii)] $\|F\|_2 \leq \frac{\sqrt{M} + \sqrt{N} + \sqrt{2\log{N}}}{\sqrt{M - \sqrt{8M\log{N}}}}$.
\end{enumerate}
\end{theorem}

\begin{proof}
Theorem~\ref{thm.normalized gaussian frames}(i) can be shown to hold with probability exceeding $1 - 2N^{-1}$ by using a bound on the norm of a Gaussian random vector in \cite[Lemma~1]{massart:annstat00} and a bound on the magnitude of the inner product of two independent Gaussian random vectors in \cite[Lemma~6]{haupt:tit10}. Specifically, pick any two distinct indices $i, j \in \{1,\dots,N\}$, and define probability events $\mathcal{E}_1 := \{|\langle g_i,g_j\rangle| \leq \delta_1\}$, $\mathcal{E}_2 := \{\|g_i\|^2 \geq M(1 - \delta_2)\}$, and $\mathcal{E}_3 := \{\|g_j\|^2 \geq M(1 - \delta_2)\}$ for $\smash{\delta_1 = \sqrt{15M\log{N}}}$ and $\smash{\delta_2 = \sqrt{(12\log{N})/M}}$. Then it follows from the union bound that
\begin{equation*}
\Pr\bigg(|\langle f_i,f_j\rangle| > \tfrac{\delta_1}{M(1-\delta_2)} \bigg)
= \Pr\bigg(\tfrac{|\langle g_i,g_j\rangle|}{\|g_i\|~\|g_j\|} > \tfrac{\delta_1}{M(1-\delta_2)} \bigg)
\leq \Pr(\mathcal{E}_1^\mathrm{c}) + \Pr(\mathcal{E}_2^\mathrm{c}) + \Pr(\mathcal{E}_3^\mathrm{c}).
\end{equation*}
One can verify that $\Pr(\mathcal{E}_2^\mathrm{c}) = \Pr(\mathcal{E}_3^\mathrm{c}) \leq N^{-3}$ because of \cite[Lemma~1]{massart:annstat00}, and we further have $\Pr(\mathcal{E}_1^\mathrm{c}) \leq 2N^{-3}$ because of \cite[Lemma~6]{haupt:tit10} and the fact that $M \geq 60\log{N}$. 
Thus, for any fixed $i$ and $j$, 
$\smash{|\langle f_i,f_j\rangle| \leq \!\sqrt{15\log{N}}/(\!\sqrt{M} - \!\sqrt{12\log{N}})}$
with probability exceeding $1 - 4N^{-3}$. 
It therefore follows by taking a union bound over all $\smash{\binom{N}{2}}$ choices for $i$ and $j$ that Theorem~\ref{thm.normalized gaussian frames}(i) holds with probability exceeding $1 - 2N^{-1}$.

Theorem~\ref{thm.normalized gaussian frames}(ii) can be shown to hold with probability exceeding $1 - 6N^{-1}$ by appealing to the preceding analysis and Hoeffding's inequality for a sum of independent, bounded random variables \cite{hoeffding:jasa63}. Specifically, fix any index $i \in \{1,\dots,N\}$, and define random variables $Z^i_j := \frac{1}{N-1}\langle f_i, f_j\rangle$. Next, define the probability event 
\begin{equation*}
\mathcal{E}_4 := \bigcap_{\substack{j=1\\j\neq i}}^N\bigg\{|Z^i_j| \leq \tfrac{1}{N-1}~\tfrac{\sqrt{15\log{N}}}{\sqrt{M} - \sqrt{12\log{N}}}\bigg\}. 
\end{equation*}
Using the analysis for the worst-case coherence of $F$ and taking a union bound over the $N-1$ possible $j$'s gives $\Pr(\mathcal{E}_4^\mathrm{c}) \leq 4N^{-2}$. 
Furthermore, taking $\delta_3 := \sqrt{15\log{N}}/(M - \sqrt{12M\log{N}})$, then elementary probability analysis gives
\begin{equation}
\label{pfeqn:gaussian_avc1}
\Pr\bigg(\Big|\sum_{j \not= i} Z^i_j\Big| > \delta_3\bigg) 
\leq \Pr\Bigg(\Big|\sum_{j \not= i} Z^i_j\Big| > \delta_3 ~\Bigg|~ \mathcal{E}_4\Bigg) + \Pr(\mathcal{E}_4^\mathrm{c})
\leq \int_{S^{M-1}} \!\!\! \Pr\Bigg(\Big|\sum_{j \not= i} Z^i_j\Big| > \delta_3 ~\Bigg|~ \mathcal{E}_4, f_i=x\Bigg)~p_{f_i}(x)~\mathrm{dH}^{M-1}(x) + 4N^{-2},
\end{equation}
where $S^{M-1}$ denotes the unit hypersphere in $\mathbb{R}^M$, $\mathrm{H}^{M-1}$ denotes the $(M-1)$-dimensional Hausdorff measure on $S^{M-1}$, and $p_{f_i}(x)$ denotes the probability density function for the random vector $f_i$. The first thing to note here is that the random variables $\{Z^i_j: j\not=i\}$ are bounded and jointly independent when conditioned on $\mathcal{E}_4$ and $f_i$. This assertion mainly follows from Bayes' rule and the fact that $\{f_j:j\not=i\}$ are jointly independent when conditioned on $f_i$. The second thing to note is that $\smash{\mathbb{E}[Z^i_j~|~\mathcal{E}_4, f_i] = 0}$ for every $j\neq i$. This comes from the fact that the random vectors $\smash{\{f_n\}_{n=1}^N}$ are independent and have a uniform distribution over $\smash{S^{M-1}}$, which in turn guarantees that the random variables $\{Z^i_j: j\not=i\}$ have a symmetric distribution around zero when conditioned on $\mathcal{E}_4$ and $f_i$. We can therefore make use of Hoeffding's inequality \cite{hoeffding:jasa63} to bound the probability expression inside the integral in \eqref{pfeqn:gaussian_avc1} as
\begin{equation}
\label{eq.prob sum bound}
\Pr\Bigg(\Big|\sum_{j \not= i} Z^i_j\Big| > \delta_3 ~\Bigg|~ \mathcal{E}_4, f_i=x\Bigg)
\leq 2\mathrm{e}^{-(N-1)/2M},
\end{equation}
which is bounded above by $2N^{-2}$ provided $\smash{M \leq \frac{N-1}{4\log{N}}}$. We can now substitute \eqref{eq.prob sum bound} into \eqref{pfeqn:gaussian_avc1} and take the union bound over the $N$ possible choices for $i$ to conclude that Theorem~\ref{thm.normalized gaussian frames}(ii) holds with probability exceeding $1 - 6N^{-1}$.

Lastly, Theorem~\ref{thm.normalized gaussian frames}(iii) can be shown to hold with probability exceeding $1 - 3N^{-1}$ by using a bound on the spectral norm of standard Gaussian random matrices reported in \cite{rudelson:icm10} along with \cite[Lemma~1]{massart:annstat00}. Specifically, define an $N \times N$ diagonal matrix $D := \mathrm{diag}(\|g_1\|^{-1}, \dots, \|g_N\|^{-1})$, and note that the entries of $G := F D^{-1}$ are independently and normally distributed with zero mean and unit variance. We therefore have from (2.3) in \cite{rudelson:icm10} that
\begin{equation}
\label{pfeqn:gaussian_spnorm1}
    \Pr\Big(\|G\|_2 > \sqrt{M} + \sqrt{N} + \sqrt{2\log{N}}\Big) \leq 2 N^{-1}.
\end{equation}
In addition, we can appeal to the preceding analysis for the probability bound on Theorem~\ref{thm.normalized gaussian frames}(i) and conclude using \cite[Lemma~1]{massart:annstat00} and a union bound over the $N$ possible choices for $i$ that
\begin{equation}
\label{pfeqn:gaussian_spnorm2}
    \Pr\Big(\|D\|_2 > \Big(M - \sqrt{8M\log{N}}\Big)^{-1/2}\Big) \leq N^{-1}.
\end{equation}
Finally, since $\|F\|_2 \leq \|G\|_2 \|D\|_2$, we can take a union bound over \eqref{pfeqn:gaussian_spnorm1} and \eqref{pfeqn:gaussian_spnorm2} to argue that Theorem~\ref{thm.normalized gaussian frames}(iii) holds with probability exceeding $1 - 3N^{-1}$. 

The complete result now follows by taking a union bound over the failure probabilities for the conditions (i)-(iii) in Theorem~\ref{thm.normalized gaussian frames}.
\end{proof}

\begin{example}
To illustrate the bounds in Theorem~\ref{thm.normalized gaussian frames}, we ran simulations in MATLAB.
Picking $N=50000$, we observed $30$ realizations of normalized Gaussian frames for each $M=700,900,1100$.
The distributions of $\mu_F$, $\nu_F$, and $\|F\|_2$ were rather tight, so we only report the ranges of values attained, along with the bounds given in Theorem~\ref{thm.normalized gaussian frames}:
\begin{equation*}
\begin{array}{rrcll}
M=700: &  \qquad\mu_F & \in & [0.1849,0.2072]                & \qquad\leq0.8458 \\
       &  \qquad\nu_F & \in & [0.5643,0.6613]\times10^{-3}  & \qquad\leq0.0320 \\
       &  \qquad\|F\|_2 & \in & [8.0521,8.0835]              & \qquad\leq11.9565\\
\\
M=900: &  \qquad\mu_F & \in & [0.1946,0.2206]                & \qquad\leq0.6848 \\
       &  \qquad\nu_F & \in & [0.5800,0.7501]\times10^{-3}  & \qquad\leq0.0229 \\
       &  \qquad\|F\|_2 & \in & [8.4352,8.4617]              & \qquad\leq10.3645\\
\\
M=1100: &  \qquad\mu_F & \in & [0.1807,0.1988]                & \qquad\leq0.5852 \\
       &  \qquad\nu_F & \in & [0.5260,0.6713]\times10^{-3}  & \qquad\leq0.0177 \\
       &  \qquad\|F\|_2 & \in & [7.7262,7.7492]              & \qquad\leq9.2927
\end{array}
\end{equation*}
These simulations seem to indicate that our bounds on $\mu_F$ and $\|F\|_2$ reflect real-world behavior, at least within an order of magnitude, whereas the bound on $\nu_F$ is rather loose.
\end{example}

\subsection{Random harmonic frames}

Random harmonic frames, constructed by randomly selecting rows of a discrete Fourier transform (DFT) matrix and normalizing the resulting columns, have received considerable attention lately in the compressed sensing literature \cite{candes:tit06a,candes:tit06b,vershynin:cpam08}.
However, to the best of our knowledge, there is no result in the literature that shows that random harmonic frames have small worst-case coherence.
To fill this gap, the following theorem characterizes the spectral norm and the worst-case and average coherence of random harmonic frames.

\begin{theorem}[Geometry of random harmonic frames]
\label{thm.random harmonic frames}
Let $U$ be an $N\times N$ non-normalized discrete Fourier transform matrix, explicitly, $U_{k\ell}:= \mathrm{e}^{2\pi\mathrm{i}k\ell/N}$ for each $k,\ell=0,\ldots,N-1$.
Next, let $\{B_i\}_{i=0}^{N-1}$ be a collection of independent Bernoulli random variables with mean $\smash{\frac{M}{N}}$, and take $\mathcal{M}:=\{i:B_i=1\}$.
Finally, construct an $|\mathcal{M}|\times N$ harmonic frame $F$ by collecting rows of $U$ which correspond to indices in $\mathcal{M}$ and normalize the columns.
Then $F$ is a unit norm tight frame: $\smash{\|F\|_2^2=\frac{N}{|\mathcal{M}|}}$.
Furthermore, provided $\smash{16\log{N}\leq M\leq \frac{N}{3}}$, the following inequalities simultaneously hold with probability exceeding $1 - 4N^{-1} - N^{-2}$:
\begin{enumerate}
\item[(i)] $\frac{1}{2}M \leq |\mathcal{M}| \leq \frac{3}{2}M$,
\item[(ii)] $\nu_F\leq\frac{\mu_F}{\sqrt{|\mathcal{M}|}}$,
\item[(iii)] $\mu_F \leq \sqrt{\frac{118(N-M)\log{N}}{MN}}$.
\end{enumerate}
\end{theorem}

\begin{proof}
The claim that $F$ is tight follows trivially from the fact that the rows of $U$ are orthogonal and that the rows of $F$ correspond to a subset of the rows of $U$. 
Next, we define the probability events $\smash{\mathcal{E}_1 := \{|\mathcal{M}| \leq \tfrac{3}{2}M\}}$ and $\smash{\mathcal{E}_2 := \{|\mathcal{M}| \geq \tfrac{1}{2}M\}}$, and claim that $\smash{\Pr(\mathcal{E}_1^\mathrm{c} \cup \mathcal{E}_2^\mathrm{c}) \leq N^{-1} + N^{-2}}$. The proof of this claim follows from a Bernstein-like large deviation inequality. 
Specifically, note that $\smash{|\mathcal{M}| = \sum_{i=0}^{N-1} B_i}$ with $\mathbb{E}[|\mathcal{M}|] = M$, and so we have from \cite[Theorem~A.1.12, Theorem~A.1.13]{alon:00} and \cite[pp.~4]{vershynin:cpam08} that for any $\delta_1 \in [0,1)$,
\begin{equation}
\label{pfeqn:size_dft}
\Pr\Big(|\mathcal{M}| > (1 + \delta_1)M\Big) 
\leq \mathrm{e}^{-M\delta_1^2(1-\delta_1)/2}
\qquad\mbox{and}\qquad
\Pr\Big(|\mathcal{M}| < (1 - \delta_1)M\Big) 
\leq \mathrm{e}^{-M\delta_1^2/2}.
\end{equation}
Taking $\delta_1 := \tfrac{1}{2}$, then a union bound gives $\Pr(\mathcal{E}_1^\mathrm{c} \cup \mathcal{E}_2^\mathrm{c}) \leq N^{-1} + N^{-2}$ provided $M \geq 16 \log{N}$. 
Conditioning on $\mathcal{E}_1 \cap \mathcal{E}_2$, we have that Theorem~\ref{thm.random harmonic frames}(i) holds trivially, while Theorem~\ref{thm.random harmonic frames}(ii) follows from Lemma~\ref{lem.sufficient conditions}. 
Specifically, we have that $\frac{N}{3} \geq M$ guarantees $N \geq 2|\mathcal{M}|$ because of the conditioning on $\mathcal{E}_1 \cap \mathcal{E}_2$, which in turn implies that $F$ satisfies either condition (i) or (ii) of Lemma~\ref{lem.sufficient conditions}, depending on whether $0\in\mathcal{M}$. 
This therefore establishes that Theorem~\ref{thm.random harmonic frames}(i)-(ii) simultaneously hold with probability exceeding $1 - N^{-1} - N^{-2}$.

The only remaining claim is that $\mu_X \leq \delta_2 := \sqrt{(118(N-M)\log{N})/MN}$ with high probability. 
To this end, define $p:=\frac{M}{N}$, and pick any two distinct indices $i, j \in \{0,\dots,N-1\}$.
Note that
\begin{equation}
\label{eq.ip of random dft fs}
\langle f_i,f_j\rangle 
=\tfrac{1}{|\mathcal{M}|}\sum_{k=0}^{N-1} B_kU_{ki}\overline{U_{kj}} 
=\tfrac{1}{|\mathcal{M}|}\sum_{k=0}^{N-1} (B_k-p)U_{ki}\overline{U_{kj}},
\end{equation}
where the last equality follows from the fact that $U$ has orthogonal columns.
Next, we write $\smash{U_{ki}\overline{U_{kj}}=\cos(\theta_k)+\mathrm{i}\sin(\theta_k)}$ for some $\theta_k\in[0,2\pi)$.
Then applying the union bound to \eqref{eq.ip of random dft fs} and to the real and imaginary parts of $\smash{U_{ki}\overline{U_{kj}}}$ gives
\begin{align}
\Pr\Big(|\langle f_i,f_j\rangle| > \delta_2\Big) 
\nonumber
&\leq \Pr\bigg(\Big|\sum_{k=0}^{N-1} (B_k - p)U_{ki}\overline{U_{kj}}\Big| > \tfrac{M\delta_2}{2\sqrt{2}}\bigg) + \Pr\Big(|\mathcal{M}| < \tfrac{M}{2\sqrt{2}}\Big)\\
\label{pfeqn:wc_dft}
&\leq \Pr\bigg(\Big|\sum_{k=0}^{N-1} (B_k - p)\cos(\theta_k) \Big| > \tfrac{M\delta_2}{4}\bigg) + \Pr\bigg(\Big|\sum_{k=0}^{N-1} (B_k - p)\sin(\theta_k) \Big| > \tfrac{M\delta_2}{4}\bigg) + N^{-3},
\end{align}
where the last term follows from \eqref{pfeqn:size_dft} and the fact that $M \geq 16 \log{N}$. 
Define random variables $Z_k:=(B_k-p)\cos(\theta_k)$.
Note that the $Z_k$'s have zero mean and are jointly independent.
Also, the $Z_k$'s are bounded by $1-p$ almost surely since $|(B_k-p)\cos(\theta_k)|\leq\max\{p,1-p\}$ and $N\geq 2M$.
Moreover, the variance of each $Z_k$ is bounded: $\mathrm{var}(Z_\ell)\leq p(1-p)$.
Therefore, we may use the Bernstein inequality for a sum of independent, bounded random variables \cite{bennett:jasa62} to bound the probability that $|\sum_{k=0}^{N-1} Z_k|$ deviates from $\delta_3 := \frac{M\delta_2}{4}$:
\begin{equation*}
\Pr\bigg(\Big|\sum_{k=0}^{N-1} (B_k - p)\cos(\theta_k) \Big| > \delta_3\bigg)
\leq 2\mathrm{e}^{-\delta_3^2/(2Np(1-p) + 2(1-p)\delta_3/3)} \leq 2N^{-3}.
\end{equation*}
Similarly, the probability that $|\sum_{k=0}^{N-1} (B_k - p)\sin(\theta_k)| > \delta_3$ is also bounded above by $2N^{-3}$. 
Substituting these probability bounds into \eqref{pfeqn:wc_dft} gives $|\langle f_i,f_j\rangle| > \delta_2$ with probability at most $5N^{-3}$ provided $M \geq 16\log{N}$. 
Finally, we take a union bound over the $\smash{\binom{N}{2}}$ possible choices for $i$ and $j$ to get that Theorem~\ref{thm.random harmonic frames}(iii) holds with probability exceeding $1 - 3N^{-1}$. 

The result now follows by taking a final union bound over $\mathcal{E}_1^\mathrm{c} \cup \mathcal{E}_2^\mathrm{c}$ and $\{\mu_X > \delta_2\}$.
\end{proof}

As stated earlier, random harmonic frames are not new to sparse signal processing.
Interestingly, for the application of compressed sensing, \cite{candes:tit05,vershynin:cpam08}  provides performance guarantees for both random harmonic and Gaussian frames, but requires more rows in a random harmonic frame to accommodate the same level of sparsity.
This suggests that random harmonic frames may be inferior to Gaussian frames as compressed sensing matrices, but practice suggests otherwise \cite{donoho:ptrsa09}.
In a sense, Theorem~\ref{thm.random harmonic frames} helps to resolve this gap in understanding; there exist compressed sensing algorithms whose performance is dictated by worst-case coherence \cite{bajwa:jcn10,donoho:tit06b,tropp:tit04,tropp:acha08}, and Theorem~\ref{thm.random harmonic frames} states that random harmonic frames have near-optimal worst-case coherence, being on the order of the Welch bound with an additional $\sqrt{\log N}$ factor.

\begin{example}
To illustrate the bounds in Theorem~\ref{thm.random harmonic frames}, we ran simulations in MATLAB.
Picking $N=5000$, we observed $30$ realizations of random harmonic frames for each $M=1000,1250,1500$.
The distributions of $|\mathcal{M}|$, $\nu_F$, and $\mu_F$ were rather tight, so we only report the ranges of values attained, along with the bounds given in Theorem~\ref{thm.random harmonic frames}.
Notice that Theorem~\ref{thm.random harmonic frames} gives a bound on $\nu_F$ in terms of both $|\mathcal{M}|$ and $\mu_F$.
To simplify matters, we show that $\smash{\nu_F\leq\frac{\min\mu_F}{\sqrt{\max|\mathcal{M}|}}\leq\frac{\mu_F}{\sqrt{|\mathcal{M}|}}}$, where the minimum and maximum are taken over all realizations in the sample:
\begin{equation*}
\begin{array}{rrcll}
M=1000: &  \qquad|\mathcal{M}| & \in & [961,1052]                & \qquad\subseteq[500,1500] \\
        &  \qquad\nu_F & \in & [0.2000,0.8082]\times10^{-3}  & \qquad\leq0.0023\approx\tfrac{0.0746}{\sqrt{1052}} \\
        &  \qquad\mu_F & \in & [0.0746,0.0890]                & \qquad\leq0.8967 \\
\\
M=1250: &  \qquad|\mathcal{M}| & \in & [1207,1305]                & \qquad\subseteq[625,1875] \\
        &  \qquad\nu_F & \in & [0.2000,0.6273]\times10^{-3}  & \qquad\leq0.0018\approx\tfrac{0.0623}{\sqrt{1305}} \\
        &  \qquad\mu_F & \in & [0.0623,0.0774]                & \qquad\leq0.7766 \\
\\
M=1500: &  \qquad|\mathcal{M}| & \in & [1454,1590]                & \qquad\subseteq[750,2250] \\
        &  \qquad\nu_F & \in & [0.2000,0.4841]\times10^{-3}  & \qquad\leq0.0015\approx\tfrac{0.0571}{\sqrt{1590}} \\
        &  \qquad\mu_F & \in & [0.0571,0.0743]                & \qquad\leq0.6849
\end{array}
\end{equation*}
The reader may have noticed how consistently the average coherence value of $\nu_F\approx0.2000\times10^{-3}$ was realized.
This occurs precisely when the zeroth row of the DFT is not selected, as the frame elements sum to zero in this case:
\begin{equation*}
\nu_F
:=\tfrac{1}{N-1}\max_{i\in\{1,\ldots,N\}}\bigg|\sum_{\substack{j=1\\j\neq i}}^N\langle f_i,f_j\rangle\bigg|
=\tfrac{1}{N-1}\max_{i\in\{1,\ldots,N\}}\bigg|\bigg\langle f_i,\sum_{j=1}^Nf_j\bigg\rangle-\|f_i\|^2\bigg|
=\tfrac{1}{N-1}.
\end{equation*}
These simulations seem to indicate that our bounds on $|\mathcal{M}|$, $\nu_F$, and $\mu_F$ leave room for improvement.
The only bound that lies within an order of magnitude of real-world behavior is our bound on $|\mathcal{M}|$.
\end{example}

\subsection{Gabor and chirp frames}
Gabor frames constitute an important class of frames, as they appear in a variety of applications such as radar \cite{herman:sp09}, speech processing \cite{wolfe:spaa01}, and quantum information theory \cite{scott:jmp10}.
Given a nonzero seed function $f:\mathbb{Z}_M\rightarrow\mathbb{C}$, we produce all time- and frequency-shifted versions: $f_{xy}(t):=f(t-x)\mathrm{e}^{2\pi\mathrm{i}yt/M}$, $t\in\mathbb{Z}_M$.
Viewing these shifted functions as vectors in $\mathbb{C}^M$ gives an $M\times M^2$ Gabor frame.
The following theorem characterizes the spectral norm and the worst-case and average coherence of Gabor frames generated from either a deterministic Alltop vector \cite{alltop:tit80} or a random Steinhaus vector.

\begin{theorem}[Geometry of Gabor frames]
\label{thm.gabor}
Take an Alltop function defined by $\smash{f(t):=\frac{1}{\sqrt{M}}\mathrm{e}^{2\pi\mathrm{i}t^3/M}}$, $t\in\mathbb{Z}_M$.
Also, take a random Steinhaus function defined by $\smash{g(t):=\frac{1}{\sqrt{M}}\mathrm{e}^{2\pi\mathrm{i}\theta_t}}$, $t\in\mathbb{Z}_M$, where the $\theta_t$'s are independent random variables distributed uniformly on the unit interval.
Then the $M\times M^2$ Gabor frames $F$ and $G$ generated by $f$ and $g$, respectively, are unit norm and tight, that is, $\|F\|_2=\|G\|_2=\sqrt{M}$, and both frames have average coherence $\smash{\leq\frac{1}{M+1}}$.
Furthermore, if $M\geq5$ is prime, then $\smash{\mu_F=\frac{1}{\sqrt{M}}}$, while if $M\geq13$, then $\mu_G\leq\sqrt{(13\log{M})/M}$ with probability exceeding $1 - 4M^{-1}$.
\end{theorem}
\begin{proof}
The tightness claim follows from \cite{lawrence:jfaa05}, in which it was shown that Gabor frames generated by nonzero seed vectors are tight. The bound on average coherence is a consequence of \cite[Theorem~7]{bajwa:jcn10} concerning arbitrary Gabor frames. The claim concerning $\mu_F$ follows directly from \cite{strohmer:acha03}, while the claim concerning $\mu_G$ is a simple consequence of \cite[Theorem~5.1]{pfander:tsp08}.
\end{proof}

Instead of taking all translates and modulates of a seed function, \cite{CF06} constructs \emph{chirp frames} by taking all powers and modulates of a chirp function.
Picking $M$ to be prime, we start with a chirp function $h_M:\mathbb{Z}_M\rightarrow\mathbb{C}$ defined by $\smash{h_M(t):=\mathrm{e}^{\pi\mathrm{i}t(t-M)/M}}$, $t\in\mathbb{Z}_M$.
The $M^2$ frame elements are then defined entrywise by $\smash{h_{ab}(t):=\frac{1}{\sqrt{M}}h_M(t)^a\mathrm{e}^{2\pi\mathrm{i}bt/M}}$, $t\in\mathbb{Z}_M$.
Certainly, chirp frames are, at the very least, similar in spirit to Gabor frames.
As a matter of fact, the chirp frame is in some sense equivalent to the Gabor frame generated by the Alltop function:
it is easy to verify that $h_{(-6x,y-3x^2)}(t)=\mathrm{e}^{2\pi\mathrm{i}(t^3+x^3)/M}f_{xy}(t)$, and when $M\geq 5$, the map $(x,y)\mapsto(-6x,y-3x^2)$ is a permutation over $\mathbb{Z}_M^2$.
Using terminology from Definition~\ref{def.flipping and wiggling}, we say the chirp frame is \emph{wiggling equivalent} to a unitary rotation of permuted Alltop Gabor frame elements.
As such, by Lemma~\ref{lem:geom_eqframes}, the chirp frame has the same spectral norm and worst-case coherence as the Alltop Gabor frame, but the average coherence may be different.
In this case, the average coherence still satisfies (SCP-2).
Indeed, adding the frame elements gives
\begin{equation*}
\sum_{a=0}^{M-1}\sum_{b=0}^{M-1}h_{ab}(t)
=\tfrac{1}{\sqrt{M}}\sum_{a=0}^{M-1}h_M(t)^a\sum_{b=0}^{M-1}\mathrm{e}^{2\pi\mathrm{i}bt/M}
=\tfrac{1}{\sqrt{M}}\sum_{a=0}^{M-1}h_M(t)^aM\delta_0(t)
=\sqrt{M}\bigg(\sum_{a=0}^{M-1}h_M(0)^a\bigg)~\delta_0(t)
=M^{3/2}\delta_0(t),
\end{equation*}
and so $\langle h_{a'b'},\sum_{a=0}^{M-1}\sum_{b=0}^{M-1}h_{ab}\rangle=\langle h_{a'b'},M^{3/2}\delta_0\rangle=M^{3/2}h_{a'b'}(0)=M=\frac{M^2}{M}$.
Therefore, Lemma~\ref{lem.sufficient conditions}(i) gives the result:

\begin{theorem}[Geometry of chirp frames]
\label{thm.chirp}
Pick $M$ prime, and let $H$ be the $M\times M^2$ frame of all powers and modulates of the chirp function $f_M$.
Then $H$ is a unit norm tight frame with $\|H\|_2=\sqrt{M}$, and has worst case coherence $\smash{\mu_H=\frac{1}{\sqrt{M}}}$ and average coherence $\smash{\nu_H\leq\frac{\mu_H}{\sqrt{M}}}$.
\end{theorem}

\begin{example}
To illustrate the bounds in Theorems~\ref{thm.gabor} and~\ref{thm.chirp}, we consider the examples of an Alltop Gabor frame and a chirp frame, each with $M=5$.
In this case, the Gabor frame has $\smash{\nu_F\approx0.1348\leq0.1667\approx\frac{1}{M+1}}$, while the chirp frame has $\smash{\nu_H=\frac{1}{6}\leq\frac{1}{5}=\frac{\mu_H}{\sqrt{M}}}$.
Note the Gabor and chirp frames have different average coherences despite being equivalent in some sense.
For the random Steinhaus Gabor frame, we ran simulations in MATLAB and observed $30$ realizations for each $M=60,70,80$.
The distributions of $\nu_G$ and $\mu_G$ were rather tight, so we only report the ranges of values attained, along with the bounds given in Theorem~\ref{thm.gabor}:
\begin{equation*}
\begin{array}{rrcll}
M=60:   &  \qquad\nu_G & \in & [0.3916,0.5958]\times10^{-2}  & \qquad\leq0.0164 \\
        &  \qquad\mu_G & \in & [0.3242,0.4216]                & \qquad\leq0.9419 \\
\\
M=70:   &  \qquad\nu_G & \in & [0.3151,0.4532]\times10^{-2}  & \qquad\leq0.0141 \\
        &  \qquad\mu_G & \in & [0.2989,0.3814]                & \qquad\leq0.8883 \\
\\
M=80:   &  \qquad\nu_G & \in & [0.2413,0.3758]\times10^{-2}  & \qquad\leq0.0124 \\
        &  \qquad\mu_G & \in & [0.2711,0.3796]                & \qquad\leq0.8439 
\end{array}
\end{equation*}
These simulations seem to indicate that bound on $\nu_G$ is conservative by an order of magnitude.
\end{example}

\subsection{Spherical 2-designs}

Lemma~\ref{lem.sufficient conditions}(ii) leads one to consider frames of vectors that sum to zero.
In \cite{HP03}, it is proved that real unit norm tight frames with this property make up another well-studied class of vector packings: spherical 2-designs.  To be clear, a collection of unit-norm vectors $F\subseteq\mathbb{R}^M$ is called a spherical $t$-design if, for every polynomial $g(x_1,\ldots,x_M)$ of degree at most $t$, we have
\begin{equation*}
\tfrac{1}{\mathrm{H}^{M-1}(S^{M-1})}\int_{S^{M-1}}g(x)~\mathrm{d}\mathrm{H}^{M-1}(x)=\tfrac{1}{|F|}\sum_{f\in F}g(f),
\end{equation*}
where $S^{M-1}$ is the unit hypersphere in $\mathbb{R}^M$ and $\mathrm{H}^{M-1}$ denotes the $(M-1)$-dimensional Hausdorff measure on $S^{M-1}$.
In words, vectors that form a spherical $t$-design serve as good representatives when calculating the average value of a degree-$t$ polynomial over the unit hypersphere.
Today, such designs find application in quantum state estimation \cite{HHH05}.

Since real unit norm tight frames always exist for $N\geq M+1$, one might suspect that spherical 2-designs are equally common, but this intuition is faulty---the sum-to-zero condition introduces certain issues.
For example, there is no spherical 2-design when $M$ is odd and $N=M+2$.
In \cite{M90}, spherical 2-designs are explicitly characterized by construction.
The following theorem gives a construction based on harmonic frames:

\begin{theorem}[Geometry of spherical 2-designs]
\label{thm.spherical 2-designs}
Pick $M$ even and $N\geq2M$.
Take an $\frac{M}{2}\times N$ harmonic frame $G$ by collecting rows from a discrete Fourier transform matrix according to a set of nonzero indices $\mathcal{M}$ and normalize the columns.
Let $m(n)$ denote $n$th largest index in $\mathcal{M}$, and define a real $M\times N$ frame $F$ by
\begin{equation*}
F_{k\ell}:=\left\{ \begin{array}{ll}\sqrt{\frac{2}{M}}\cos(\frac{2\pi m((k+1)/2)\ell}{N}),&k\mbox{ odd}\\\sqrt{\frac{2}{M}}\sin(\frac{2\pi m(k/2)\ell}{N}),&k\mbox{ even}\end{array} \right., \qquad k=1,\ldots,M, ~\ell=0,\ldots,N-1.
\end{equation*}
Then $F$ is unit norm and tight, i.e., $\|F\|_2^2=\frac{N}{M}$, with worst-case coherence $\mu_F\leq\mu_G$ and average coherence $\nu_F\leq\frac{\mu_F}{\sqrt{M}}$.
\end{theorem}

\begin{proof}
It is easy to verify that $F$ is a unit norm tight frame using the geometric sum formula.
Also, since the frame elements sum to zero and $N\geq 2M$, the claim regarding average coherence follows from Lemma~\ref{lem.sufficient conditions}(ii).
It remains to prove $\mu_F\leq\mu_G$.  For each distinct pair of indices $i,j\in\{1,\ldots,N\}$, we have
\begin{equation*}
\langle f_i,f_j\rangle
=\tfrac{2}{M}\sum_{m\in\mathcal{M}}\Big(\cos(\tfrac{2\pi mi}{N})\cos(\tfrac{2\pi mj}{N})+\sin(\tfrac{2\pi mi}{N})\sin(\tfrac{2\pi mj}{N})\Big)
=\tfrac{2}{M}\sum_{m\in\mathcal{M}}\cos(\tfrac{2\pi m(i-j)}{N})
=\mathrm{Re}\langle g_i,g_j\rangle,
\end{equation*}
and so $|\langle f_i,f_j\rangle|=|\mathrm{Re}\langle g_i,g_j\rangle|\leq|\langle g_i,g_j\rangle|$.
This gives the result.
\end{proof}

\begin{example}
To illustrate the bounds in Theorem~\ref{thm.spherical 2-designs}, we consider the spherical 2-design constructed from a $9\times 37$ harmonic equiangular tight frame \cite{XZG05}.
Specifically, we take a $37\times 37$ DFT matrix, choose nonzero row indices
\begin{equation*}
\mathcal{M}=\{1,7,9,10,12,16,26,33,34\},
\end{equation*}
and normalize the columns to get a harmonic frame $G$ whose worst-case coherence achieves the Welch bound: $\smash{\mu_G=\sqrt{\frac{37-9}{9(37-1)}}\approx0.2940}$.
Following Theorem~\ref{thm.spherical 2-designs}, we produce a spherical 2-design $F$ with $\mu_F\approx0.1967\leq\mu_G$ and $\smash{\nu_F\approx0.0278\leq0.0464\approx\frac{\mu_F}{\sqrt{M}}}$.
\end{example}

\subsection{Steiner equiangular tight frames}

We now consider a construction that dates back to Seidel with \cite{S73}, and was recently developed further in \cite{FMT10}.
Here, a special type of block design is used to build an equiangular tight frame (ETF), that is, a tight frame in which the modulus of every inner product between frame elements achieves the Welch bound.
Let's start with a definition:

\begin{definition}
A $(t,k,v)$-\emph{Steiner system} is a $v$-element set $V$ with a collection of $k$-element subsets of $V$, called \emph{blocks}, with the property that any $t$-element subset of $V$ is contained in exactly one block.
The $\{0,1\}$-\emph{incidence matrix} $A$ of a Steiner system has entries $A_{ij}$, where $A_{ij}=1$ if the $i$th block contains the $j$th element, and otherwise $A_{ij}=0$.
\end{definition}

One example of a Steiner system is a set with all possible two-element blocks.
This forms a $(2,2,v)$-Steiner system because every pair of elements is contained in exactly one block.
The following theorem details how \cite{FMT10} constructs ETFs using Steiner systems.

\begin{theorem}[Constructing Steiner equiangular tight frames \cite{FMT10}]
\label{thm.steiner etfs}
Every $(2,k,v)$-Steiner system can be used to build a $\smash{\frac{v(v-1)}{k(k-1)}\times v(1+\frac{v-1}{k-1})}$ equiangular tight frame $F$ according the following procedure:
\begin{enumerate}
\item[(i)] Let $A$ be the $\frac{v(v-1)}{k(k-1)}\times v$ incidence matrix of a $(2,k,v)$-Steiner system.
\item[(ii)] Let $H$ be the $(1+\frac{v-1}{k-1})\times(1+\frac{v-1}{k-1})$ discrete Fourier transform matrix.
\item[(iii)] For each $j=1,\ldots,v$, let $F_j$ be a $\frac{v(v-1)}{k(k-1)}\times(1+\frac{v-1}{k-1})$ matrix obtained from the $j$th column of $A$ by replacing each of the one-valued entries with a distinct row of $H$, and every zero-valued entry with a row of zeros.
\item[(iv)] Concatenate and rescale the $F_j$'s to form $F=(\frac{k-1}{v-1})^\frac{1}{2}[F_1\cdots F_v]$.
\end{enumerate}
\end{theorem}

As an example, we build an ETF from a (2,2,3)-Steiner system.
In this case, the incidence matrix is
\begin{equation*}
A=\left[\begin{array}{ccc}+&+&\\+&&+\\&+&+\end{array}\right].
\end{equation*}
For this matrix, each row represents a block.
Since each block contains two elements, each row of the matrix has two ones.
Also, any two elements determines a unique common row, and so any two columns have a single one in common.
To form the corresponding $3\times 9$ ETF $F$, we use the $3\times 3$ DFT matrix.
Letting $\omega=\mathrm{e}^{2\pi\mathrm{i}/3}$, we have
\begin{equation*}
H=\left[\begin{array}{lll}1&1&1\\1&\omega&\omega^2\\1&\omega^2&\omega\end{array}\right].
\end{equation*}
Finally, we replace the two ones in each column of $A$ with the second and third rows of $H$.
Normalizing the columns gives $3\times 9$ ETF:
\begin{equation}
\label{eq.steiner etf example} F=\tfrac{1}{\sqrt{2}}\left[\begin{array}{lllllllll}1&\omega&\omega^2&1&\omega&\omega^2&&&\\1&\omega^2&\omega&&&&1&\omega&\omega^2\\&&&1&\omega^2&\omega&1&\omega^2&\omega\end{array}\right].
\end{equation}

Several infinite families of $(2,k,v)$-Steiner systems are already known, and Theorem~\ref{thm.steiner etfs} says that each one can be used to build an ETF.
See \cite{FMT10} for a complete discussion of this construction and how it relates to each known family of Steiner systems.
Interestingly, every Steiner ETF satisfies $N\geq2M$.
If, in step (iii) of Theorem~\ref{thm.steiner etfs}, we choose the distinct rows to be the $\frac{v-1}{k-1}$ rows of the DFT $H$ that are not all-ones, then the sum of columns of each $F_j$ is zero, meaning the sum of columns of $F$ is also zero.
This was done in the example above, and the columns sum to zero, accordingly.
Therefore, by Lemma~\ref{lem.sufficient conditions}(ii), Steiner ETFs satisfy (SCP-2).
This gives the following theorem:

\begin{theorem}[Geometry of Steiner equiangular tight frames]
\label{thm.steiner etf coherence}
Build an $M\times N$ matrix $F$ according to Theorem~\ref{thm.steiner etfs}, and in step (iii), choose rows from the discrete Fourier transform matrix $H$ that are not all-ones.
Then $F$ is an equiangular tight frame, meaning $\|F\|_2^2=\frac{N}{M}$ and $\mu_F^2=\frac{N-M}{M(N-1)}$, and has average coherence $\nu_F\leq\frac{\mu_F}{\sqrt{M}}$.
\end{theorem}

\begin{example}
To illustrate the bound in Theorem~\ref{thm.steiner etf coherence}, we note that the example given in \eqref{eq.steiner etf example} has $\smash{\nu_F=\frac{1}{8}\leq\frac{1}{2\sqrt{3}}=\frac{\mu_F}{\sqrt{M}}}$.
\end{example}

\subsection{Code-based frames}

Many structures in coding theory are also useful in frame theory.
In this section, we build frames from a code that originally emerged with Berlekamp in \cite{B70}, and found recent reincarnation with \cite{YG06}.
We build a $2^m\times2^{(t+1)m}$ frame, indexing rows by elements of $\mathbb{F}_{2^m}$ and indexing columns by $(t+1)$-tuples of elements from $\mathbb{F}_{2^m}$.
For $x\in\mathbb{F}_{2^m}$ and $\alpha\in\mathbb{F}_{2^m}^{t+1}$, the corresponding entry of the matrix $F$ is given by
\begin{equation}
\label{eq.matrix defn}
F_{x\alpha}=\tfrac{1}{\sqrt{2^{m}}}(-1)^{\mathrm{Tr}\big[\alpha_0x+\sum_{i=1}^t\alpha_ix^{2^i+1}\big]},
\end{equation}
where $\mathrm{Tr}:\mathbb{F}_{2^m}\rightarrow\mathbb{F}_2$ denotes the trace map, defined by $\mathrm{Tr}(z)=\sum_{i=0}^{m-1}z^{2^i}$.
The following theorem gives the spectral norm and the worst-case and average coherence of this frame.

\begin{table}
\begin{center}
\begin{scriptsize}
\begin{tabular}{lllllll}
\hline
Name &$\mathbb{R}/\mathbb{C}$	&Size &$\mu_F$ &$\nu_F$ &Restrictions &Probability\\
\hline
Normalized Gaussian &$\mathbb{R}$ &$M\times N$ &$\leq \frac{\sqrt{15\log{N}}}{\sqrt{M} - \sqrt{12\log{N}}}$ &$\leq \frac{\sqrt{15\log{N}}}{M - \sqrt{12M\log{N}}}$ &$60\log N\leq M\leq\frac{N-1}{4\log N}$ &$\geq1-\frac{11}{N}$\\

Random harmonic &$\mathbb{C}$ &$|\mathcal{M}|\times N$, $\frac{1}{2}M\leq|\mathcal{M}|\leq\frac{3}{2}M$ &$\leq \sqrt{\frac{118(N-M)\log{N}}{MN}}$ &$\leq\frac{\mu_F}{\sqrt{|\mathcal{M}|}}$ &$16\log{N}\leq M\leq \frac{N}{3}$ &$\geq 1 - \frac{4}{N}-\frac{1}{N^2}$\\

Alltop Gabor &$\mathbb{C}$ &$M\times M^2$ &$=\frac{1}{\sqrt{M}}$ &$\leq\frac{1}{M+1}$ &$M\geq 5$ prime &Deterministic\\

Steinhaus Gabor &$\mathbb{C}$ &$M\times M^2$ &$\leq\sqrt{\frac{13\log M}{M}}$ &$\leq\frac{1}{M+1}$ &$M\geq13$ &$\geq1-\frac{4}{M}$\\

Chirp &$\mathbb{C}$ &$M\times M^2$ &$=\frac{1}{\sqrt{M}}$ &$\leq\frac{\mu_F}{\sqrt{M}}$ &$M$ prime &Deterministic\\

$\overset{\mbox{Spherical 2-design}}{\mbox{from harmonic }G}$ &$\mathbb{R}$ &$M\times N$ &$\leq\mu_G$ &$\leq\frac{\mu_F}{\sqrt{M}}$ &$M$ even, $N\geq 2M$ &Deterministic\\

Steiner &$\mathbb{C}$ &$M\times N$, $M=\frac{v(v-1)}{k(k-1)}$, $N=v(1+\frac{v-1}{k-1})$ &$=\sqrt{\frac{N-M}{M(N-1)}}$ &$\leq\frac{\mu_F}{\sqrt{M}}$ &$\exists(2,k,v)$-Steiner system &Deterministic\\

Code-based &$\mathbb{R}$ &$2^m\times 2^{(t+1)m}$ &$\leq\frac{1}{\sqrt{2^{m-2t-1}}}$ &$\leq\frac{\mu_F}{\sqrt{2^m}}$ &None &Deterministic\\

\hline
\end{tabular}
\end{scriptsize}
\caption{\label{table.constructions}
Eight constructions detailed in this paper.
All of these are unit norm tight frames except for the normalized Gaussian frame, which has squared spectral norm $\|F\|_2^2\leq(\!\sqrt{M}+\!\sqrt{N}+\!\sqrt{2\log{N}})^2/(M-\!\sqrt{8M\log{N}})$ in the same probability event as is measured above.
}
\end{center}
\end{table}

\begin{theorem}[Geometry of code-based frames]
\label{thm.code-based coherence}
The $2^m\times2^{(t+1)m}$ frame defined by \eqref{eq.matrix defn} is unit norm and tight, i.e., $\|F\|_2^2=2^{tm}$, with worst-case coherence $\mu_F\leq \frac{1}{\sqrt{2^{m-2t-1}}}$ and average coherence $\smash{\nu_F\leq\frac{\mu_F}{\sqrt{2^{m}}}}$.
\end{theorem}

\begin{proof}
For the tightness claim, we use the linearity of the trace map to write the inner product of rows $x$ and $y$:
\begin{equation*} \sum_{\alpha\in\mathbb{F}_{2^m}^{t+1}}\!\!\tfrac{1}{\sqrt{2^{m}}}(-1)^{\mathrm{Tr}\big[\alpha_0x+\sum_{i=1}^t\alpha_ix^{2^i+1}\big]}\tfrac{1}{\sqrt{2^{m}}}(-1)^{\mathrm{Tr}\big[\alpha_0y+\sum_{i=1}^t\alpha_iy^{2^i+1}\big]}=\tfrac{1}{2^m}\bigg(\!\sum_{\alpha_0\in\mathbb{F}_{2^m}}(-1)^{\mathrm{Tr}[\alpha_0(x+y)]}\bigg)\!\!\sum_{\alpha_1\in\mathbb{F}_{2^m}}\!\!\cdots\!\!\sum_{\alpha_t\in\mathbb{F}_{2^m}}\!\!(-1)^{\mathrm{Tr}\big[\sum_{i=1}^t\alpha_i(x^{2^i+1}+y^{2^i+1})\big]}.
\end{equation*}
This expression is $2^{tm}$ when $x=y$.  
Otherwise, note that $\alpha_0\mapsto(-1)^{\mathrm{Tr}[\alpha_0(x+y)]}\in\{\pm1\}$ defines a homomorphism on $\mathbb{F}_{2^m}$.
Since $(x+y)^{-1}\mapsto-1$, the inverse images of $\pm1$ under this homomorphism must form two cosets of equal size, and so $\sum_{\alpha_0\in\mathbb{F}_{2^m}}(-1)^{\mathrm{Tr}[\alpha_0(x+y)]}=0$, meaning distinct rows in $F$ are orthogonal.  
Thus, $F$ is a unit norm tight frame.

For the worst-case coherence claim, we first note that the linearity of the trace map gives
\begin{equation*} (-1)^{\mathrm{Tr}\big[\alpha_0x+\sum_{i=1}^t\alpha_ix^{2^i+1}\big]}(-1)^{\mathrm{Tr}\big[\alpha'_0x+\sum_{i=1}^t\alpha'_ix^{2^i+1}\big]}=(-1)^{\mathrm{Tr}\big[(\alpha_0+\alpha'_0)x+\sum_{i=1}^t(\alpha_i+\alpha'_i)x^{2^i+1}\big]},
\end{equation*}
i.e., every inner product between columns of $F$ is a sum over another column.
Thus, there exists $\alpha\in\mathbb{F}_{2^m}^{t+1}$ such that
\begin{equation*}
2^{2m}\mu_F^2
=\bigg(\sum_{x\in\mathbb{F}_{2^m}}(-1)^{\mathrm{Tr}\big[\alpha_0x+\sum_{i=1}^t\alpha_ix^{2^i+1}\big]}\bigg)^2
=2^m+\sum_{x\in\mathbb{F}_{2^m}}\sum_{\substack{y\in\mathbb{F}_{2^m}\\y\neq x}}(-1)^{\mathrm{Tr}\big[\alpha_0(x+y)+\sum_{i=1}^t\alpha_i\big((x+y)^{2^i+1}+\sum_{j=0}^{i-1}(xy)^{2^j}(x+y)^{2^i-2^{j+1}+1}\big)\big]},
\end{equation*}
where the last equality is by the identity
$(x+y)^{2^i+1}=x^{2^i+1}+y^{2^i+1}+\sum_{j=0}^{i-1}(xy)^{2^j}(x+y)^{2^i-2^{j+1}+1}$, whose proof is a simple exercise of induction.
From here, we perform a change of variables: $u:= x+y$ and $v:= xy$.
Notice that $(u,v)$ corresponds to $(x,y)$ for some $x\neq y$ whenever $(z+x)(z+y)=z^2+uz+v$ has two solutions, that is, whenever $\smash{\mathrm{Tr}(\frac{v}{u^2})=0}$.
Since $(u,v)$ corresponds to both $(x,y)$ and $(y,x)$, we must correct for under-counting:
\begin{align}
2^{2m}\mu_F^2
\nonumber
&=2^m+2\sum_{\substack{u\in\mathbb{F}_{2^m}\\u\neq0}} \sum_{\substack{v\in\mathbb{F}_{2^m}\\\mathrm{Tr}(v/u^2)=0}}(-1)^{\mathrm{Tr}\big[\alpha_0u+\sum_{i=1}^t\alpha_i\big(u^{2^i+1}+\sum_{j=0}^{i-1}v^{2^j}u^{2^i-2^{j+1}+1}\big)\big]}\\
\nonumber
&=2^m+2\sum_{\substack{u\in\mathbb{F}_{2^m}\\u\neq0}}(-1)^{\mathrm{Tr}\big[\alpha_0u+\sum_{i=1}^t\alpha_iu^{2^i+1}\big]}\sum_{\substack{v\in\mathbb{F}_{2^m}\\\mathrm{Tr}(v/u^2)=0}
}(-1)^{\mathrm{Tr}\big[\big(\sum_{i=1}^t\sum_{j=0}^{i-1}\alpha_i^{2^{-j}}u^{2^{i-j}-2+2^{-j}}\big)v\big]}\\
\label{eq.big}
&\leq2^m+2\sum_{\substack{u\in\mathbb{F}_{2^m}\\u\neq0}}~\bigg|\!\!\!\sum_{\substack{v\in\mathbb{F}_{2^m}
\\\mathrm{Tr}(v/u^2)=0}}\!\!\!(-1)^{\mathrm{Tr}[p(u)v]}~\bigg|,
\end{align}
where the second equality is by repeated application of $\mathrm{Tr}(z)=\mathrm{Tr}(z^2)$, and $\smash{p(u):=\sum_{i=1}^t\sum_{j=0}^{i-1}\alpha_i^{2^{-j}}u^{2^{i-j}-2+2^{-j}}}$.
To bound $\mu_F$, we will count the $u$'s that produce nonzero summands in \eqref{eq.big}.

For each $u\neq0,$ we have a homomorphism $\smash{\chi_u:\{v\in\mathbb{F}_{2^m}:\mathrm{Tr}(\frac{v}{u^2})=0\}\rightarrow\{\pm1\}}$ defined by
$\chi_u(v):=(-1)^{\mathrm{Tr}[p(u)v]}$.
Pick $u\neq0$ for which there exists a $v$ such that both $\smash{\mathrm{Tr}(\frac{v}{u^2})=0}$ and $\mathrm{Tr}[p(u)v]=1$.
Then $\chi_u(v)=-1$, and so the kernel of $\chi_u$ is the same size as the coset $\smash{\{v\in\mathbb{F}_{2^m}:\mathrm{Tr}(\frac{v}{u^2})=0,\chi_u(v)=-1\}}$, meaning the summand associated with $u$ in \eqref{eq.big} is zero.
Hence, the nonzero summands in \eqref{eq.big} require $\smash{\mathrm{Tr}(\frac{v}{u^2})=0}$ and $\mathrm{Tr}[p(u)v]=0$.
This is certainly possible whenever $p(u)=0$.
Exponentiation gives 
\begin{equation*}
p(u)^{2^{t-1}}=\sum_{i=1}^t\sum_{j=0}^{i-1}\alpha_i^{2^{t-j-1}}u^{2^{t+i-j-1}-2^t+2^{t-j-1}}, 
\end{equation*}
which has degree $2^{2t-1}-2^{t-1}$.
Thus, $p(u)=0$ has at most $2^{2t-1}-2^{t-1}$ solutions, and each such $u$ produces a summand in \eqref{eq.big} of size $2^{m-1}$.
Next, we consider the $u$'s for which $\smash{\mathrm{Tr}(\frac{v}{u^2})=0}$, $\mathrm{Tr}[p(u)v]=0$, and $p(u)\neq0$.
In this case, the hyperplanes defined by $\smash{\mathrm{Tr}(\frac{v}{u^2})=0}$ and $\mathrm{Tr}[p(u)v]=0$ are parallel, and so $\smash{p(u)=\frac{1}{u^2}}$.
Here, 
\begin{equation*}
1=(u^2p(u))^{2^{t-1}}=\sum_{i=1}^t\sum_{j=0}^{i-1}\alpha_i^{2^{t-j-1}}u^{2^{t+i-j-1}+2^{t-j-1}}, 
\end{equation*}
which has degree $2^{2t-1}+2^{t-1}$.
Thus, $\smash{p(u)=\frac{1}{u^2}}$ has at most $2^{2t-1}+2^{t-1}$ solutions, and each such $u$ produces a summand in \eqref{eq.big} of size $2^{m-1}$.
We can now continue the bound from \eqref{eq.big}: $2^{2m}\mu_F^2\leq 2^m+2(2^{2t-1}-2^{t-1}+2^{2t-1}+2^{t-1})2^{m-1}\leq 2^{m+2t+1}$. 
From here, isolating $\mu_F$ gives the claim.

Lastly, for the average coherence, pick some $x\in\mathbb{F}_{2^m}$.
Then summing the entries in the $x$th row gives
\begin{equation*}
\sum_{\alpha\in\mathbb{F}_{2^m}^{t+1}}\tfrac{1}{\sqrt{2^{m}}}(-1)^{\mathrm{Tr}\big[\alpha_0x+\sum_{i=1}^t\alpha_ix^{2^i+1}\big]}
=\tfrac{1}{\sqrt{2^{m}}}\bigg(\sum_{\alpha_0\in\mathbb{F}_{2^m}}(-1)^{\mathrm{Tr}(\alpha_0x)}\bigg)\sum_{\alpha_1\in\mathbb{F}_{2^m}}\cdots\sum_{\alpha_t\in\mathbb{F}_{2^m}}(-1)^{\mathrm{Tr}\big[\sum_{i=1}^t\alpha_ix^{2^i+1}\big]}
=\left\{\begin{array}{lc}2^{(t+1/2)m},&x=0\\0,&x\neq0\end{array}\right. .
\end{equation*}
That is, the frame elements sum to a multiple of an identity basis element: $\smash{\sum_{\alpha\in\mathbb{F}_{2^m}^{t+1}}f_\alpha=2^{(t+1/2)m}\delta_0}$.
Since every entry in row $x=0$ is $\smash{\frac{1}{\sqrt{2^{m}}}}$, we have  
$\smash{\langle f_{\alpha'},\sum_{\alpha\in\mathbb{F}_{2^m}^{t+1}}f_\alpha \rangle=\frac{2^{(t+1)m}}{2^m}}$
for every $\alpha'\in\mathbb{F}_{2^m}^{t+1}$, and so by Lemma~\ref{lem.sufficient conditions}(i), we are done.
\end{proof}

\begin{example}
To illustrate the bounds in Theorem~\ref{thm.code-based coherence}, we consider the example where $m=4$ and $t=1$.
This is a $16\times 256$ code-based frame $F$ with $\smash{\mu_F=\frac{1}{2}\leq\frac{1}{\sqrt{2}}=\frac{1}{\sqrt{2^{m-2t-1}}}}$ and $\smash{\nu_F=\frac{1}{17}\leq\frac{1}{8}=\frac{\mu_F}{\sqrt{2^m}}}$.
\end{example}

\section{Fundamental limits on worst-case coherence}

In many applications of frames, performance is dictated by worst-case coherence \cite{bajwa:jcn10,candes:annstat09,donoho:tit06b,HP03,mixon:icassp11,strohmer:acha03,tropp:tit04,tropp:acha08,zahedi:acc10}. 
It is therefore particularly important to understand which worst-case coherence values are achievable.
To this end, the Welch bound is commonly used in the literature.
When worst-case coherence achieves the Welch bound, the frame is equiangular and tight \cite{strohmer:acha03}; one of the biggest open problems in frame theory concerns equiangular tight frames \cite{scott:jmp10}.
However, equiangular tight frames cannot have more vectors than the square of the spatial dimension \cite{strohmer:acha03}, meaning the Welch bound is not tight whenever $N>M^2$.
When the number of vectors $N$ is exceedingly large, the following theorem gives a better bound:

\begin{theorem}[\cite{alon:dm03,nelson:jat11}]
\label{thm.asymptotic bound}
Every sufficiently large $M\times N$ unit norm frame $F$ with $N\geq2M$ and worst-case coherence $\mu_F<\frac{1}{2}$ satisfies
\begin{equation}
\mu_F^2\log\big(\tfrac{1}{\mu_F}\big)\geq\tfrac{C\log N}{M}
\end{equation}
for some constant $C>0$.
\end{theorem}

For a fixed worst-case coherence $\mu_F<\frac{1}{2}$, this bound indicates that the number of vectors $N$ cannot exceed some exponential in the spatial dimension $M$, that is, $N\leq a^M$ for some $a>0$.
However, since the constant $C$ is not established in this theorem, it is unclear which base $a$ is appropriate for each $\mu_F$.
The following theorem is a little more explicit in this regard:

\begin{theorem}[\cite{MSEA03,XZG05}]
\label{thm.complex bound}
Every $M\times N$ unit norm frame $F$ has worst-case coherence $\mu_F\geq1-2N^{-1/(M-1)}$.
Furthermore, taking $N=\Theta(a^M)$, this lower bound goes to $1-\frac{2}{a}$ as $M\rightarrow\infty$.
\end{theorem}

For many applications, it does not make sense to use a complex frame, but the bound in Theorem~\ref{thm.complex bound} is known to be loose for real frames \cite{CHS96}.
We therefore improve Theorems~\ref{thm.asymptotic bound} and~\ref{thm.complex bound} for the case of real unit norm frames:

\begin{theorem}
\label{thm.bound}
Every real $M\times N$ unit norm frame $F$ has worst-case coherence
\begin{equation}
\label{eq.bound}
\mu_F\geq\cos\bigg[\pi\Big(\tfrac{M-1}{N\pi^{1/2}}~\tfrac{\Gamma(\frac{M-1}{2})}{\Gamma(\frac{M}{2})}\Big)^{\frac{1}{M-1}}\bigg].
\end{equation}
Furthermore, taking $N=\Theta(a^M)$, this lower bound goes to $\cos(\frac{\pi}{a})$ as $M\rightarrow\infty$.
\end{theorem}

Before proving this theorem, we first consider the special case where the spatial dimension is $M=3$:

\begin{lemma}
\label{lem.3d points}
Given $N$ points on the unit sphere $S^{2}\subseteq\mathbb{R}^3$, the smallest angle between points is $\leq2\cos^{-1}\big(1-\frac{2}{N}\big)$.
\end{lemma}

\begin{proof}
We first claim there exists a closed spherical cap in $S^{2}$ with area $\smash{\frac{4\pi}{N}}$ that contains two of the $N$ points.  
Suppose otherwise, and take $\gamma$ to be the angular radius of a spherical cap with area $\smash{\frac{4\pi}{N}}$.  
That is, $\gamma$ is the angle between the center of the cap and every point on the boundary.  
Since the cap is closed, we must have that the smallest angle $\alpha$ between any two of our $N$ points satisfies $\alpha>2\gamma$.  
Let $C(p,\theta)$ denote the closed spherical cap centered at $p\in S^2$ of angular radius $\theta$, and let $P$ denote our set of $N$ points.  
Then we know for $p\in P$, the $C(p,\gamma)$'s are disjoint, $\frac{\alpha}{2}>\gamma$, and $\bigcup_{p\in P}C(p,\tfrac{\alpha}{2})\subseteq S^2$, and so taking 2-dimensional Hausdorff measures on the sphere gives
\begin{equation*} 
\mathrm{H}^2(S^2)=4\pi=\mathrm{H}^2\bigg(\bigcup_{p\in P}C(p,\gamma)\bigg)<\mathrm{H}^2\bigg(\bigcup_{p\in P}C(p,\tfrac{\alpha}{2})\bigg)\leq\mathrm{H}^2(S^2),
\end{equation*} 
a contradiction.  

Since two of the points reside in a spherical cap of area $\smash{\frac{4\pi}{N}}$, we know $\alpha$ is no more than twice the radius of this cap.  
We use spherical coordinates to relate the cap's area to the radius:
$\smash{\mathrm{H}^2(C(\cdot,\gamma))=2\pi\int_0^\gamma\sin\phi~\mathrm{d}\phi=2\pi(1-\cos\gamma)}$. 
Therefore, when $\smash{\mathrm{H}^2(C(\cdot,\gamma))=\frac{4\pi}{N}}$, we have $\gamma=\cos^{-1}(1-\frac{2}{N})$, and so $\alpha\leq2\gamma$ gives
the result.
\end{proof}

\begin{theorem}
\label{thm.3d points}
Every real $3\times N$ unit norm frame $F$ has worst-case coherence $\mu_F\geq1-\frac{4}{N}+\frac{2}{N^2}$.
\end{theorem}

\begin{proof}
Packing $N$ unit vectors in $\mathbb{R}^3$ corresponds to packing $2N$ antipodal points in $S^2$, 
and so Lemma~\ref{lem.3d points} gives $\alpha\leq2\cos^{-1}(1-\frac{1}{N})$.
Applying the double angle formula to $\mu_F=\cos\alpha\geq\cos[2\cos^{-1}(1-\frac{1}{N})]$ gives the result.
\end{proof}

\begin{figure}[t]
\centering
\begin{picture}(320,193)(0,0)
\put(0,0){\includegraphics[width=4.5in]{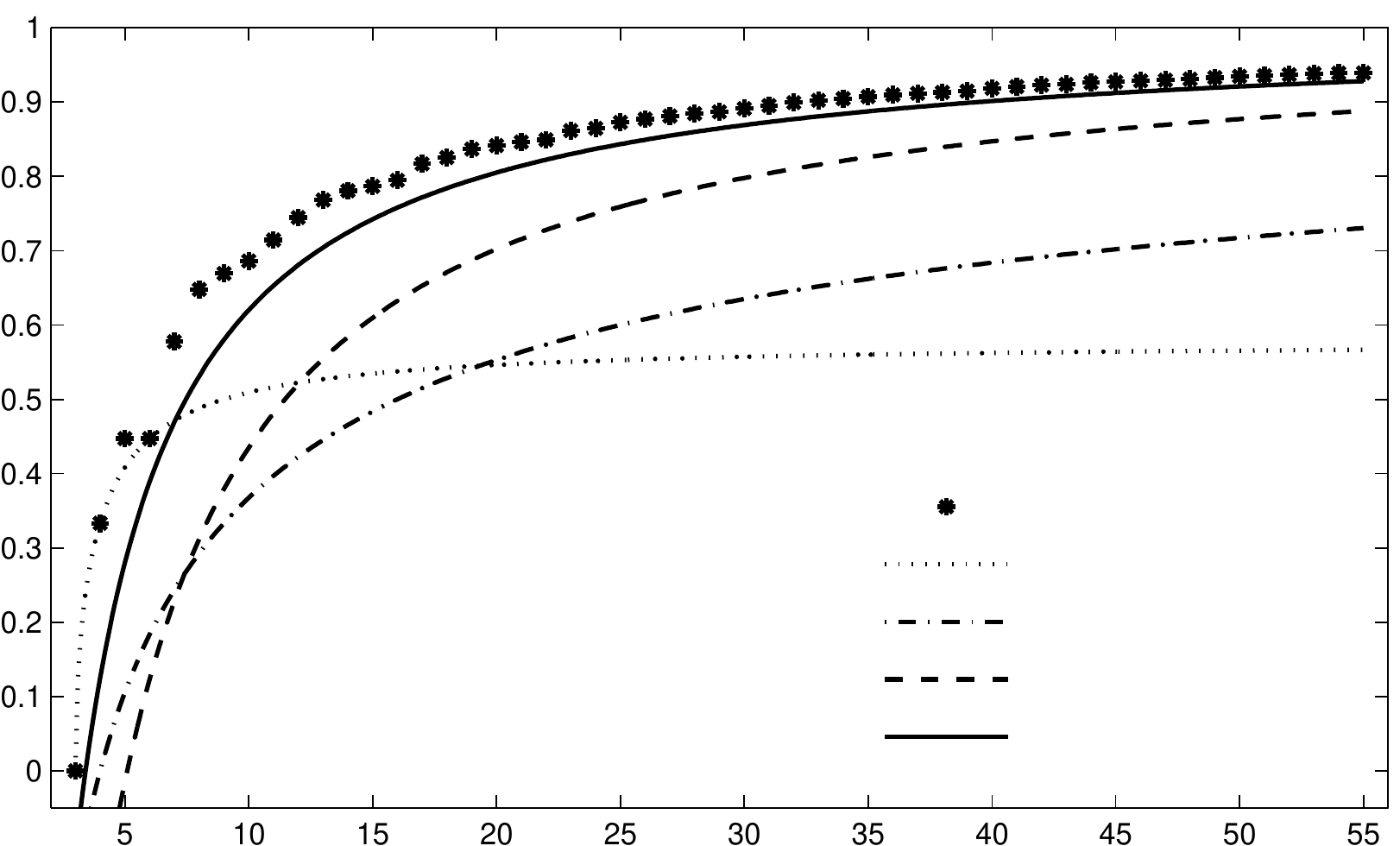}}

\put(330,0){$N$}
\put(-12,185){$\mu_F$}

\put(240,76){\footnotesize{Numerically optimal}}
\put(240,63){\footnotesize{Welch bound}}
\put(240,49){\footnotesize{Theorem~\ref{thm.complex bound}}}
\put(240,36){\footnotesize{Theorem~\ref{thm.bound}}}
\put(240,23){\footnotesize{Theorem~\ref{thm.3d points}}}

\end{picture}

\caption{Different bounds on worst-case coherence for $M=3$, $N=3,\ldots,55$.
Stars give numerically determined optimal worst-case coherence of $N$ real unit vectors, found in \cite{CHS96}.
Dotted curve gives Welch bound, dash-dotted curve gives bound from Theorem~\ref{thm.complex bound}, dashed curve gives bound from Theorem~\ref{thm.bound}, and solid curve gives bound from Theorem~\ref{thm.3d points}.
 \label{figure}}
\end{figure}
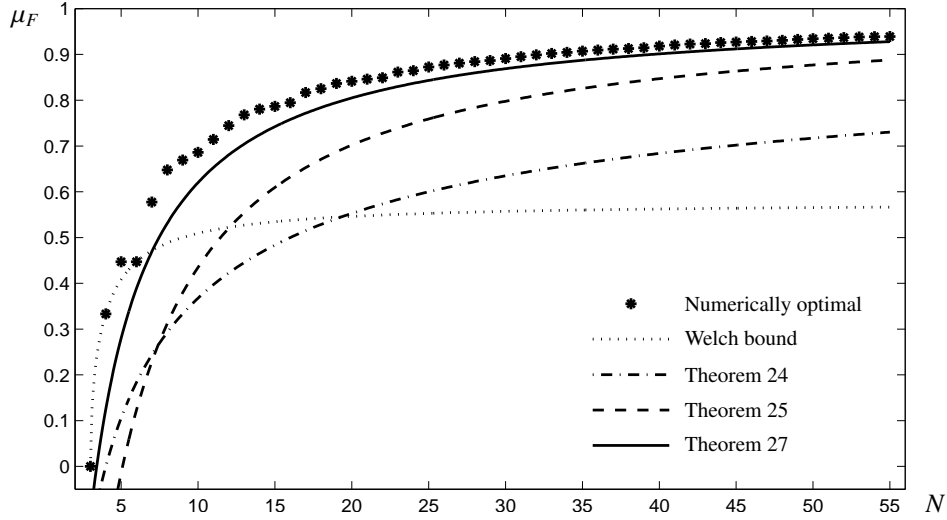

Now that we understand the special case where $M=3$, we tackle the general case:

\begin{proof}[Proof of Theorem~\ref{thm.bound}]
As in the proof of Theorem~\ref{thm.3d points}, we relate packing $N$ unit vectors to packing $2N$ points in the hypersphere $S^{M-1}\subseteq\mathbb{R}^M$.
The argument in the proof of Lemma~\ref{lem.3d points} generalizes so that two of the $2N$ points must reside in some closed hyperspherical cap of hypersurface area $\frac{1}{2N}\mathrm{H}^{M-1}(S^{M-1})$.
Therefore, the smallest angle $\alpha$ between these points is no more than twice the radius of this cap.
Let $C(\gamma)$ denote a hyperspherical cap of angular radius $\gamma$.
Then we use hyperspherical coordinates to get
\begin{align}
\nonumber
\mathrm{H}^{M-1}(C(\gamma))
&=\int_{\phi_1=0}^\gamma\int_{\phi_2=0}^\pi\cdots\int_{\phi_{M-2}=0}^\pi\int_{\phi_{M-1}=0}^{2\pi}\sin^{M-2}(\phi_1)\cdots\sin^1(\phi_{M-2})~\mathrm{d}\phi_{M-1}\cdots\mathrm{d}\phi_1\\
\nonumber
&=2\pi\bigg(\prod_{j=1}^{M-3}\pi^{1/2}\tfrac{\Gamma(\frac{j+1}{2})}{\Gamma(\frac{j}{2}+1)}\bigg)\int_0^\gamma\sin^{M-2}\phi ~\mathrm{d}\phi\\
\label{eq.gamma}
&=\tfrac{2\pi^{(M-1)/2}}{\Gamma(\frac{M-1}{2})}\int_0^\gamma\sin^{M-2}\phi ~\mathrm{d}\phi.
\end{align}
We wish to solve for $\gamma$, but analytically inverting $\int_0^\gamma\sin^{M-2}\phi ~\mathrm{d}\phi$ is difficult.
Instead, we use $\sin\phi\geq\frac{2\phi}{\pi}$ for $\phi\in[0,\frac{\pi}{2}]$.
Note that we do not lose generality by forcing $\gamma\leq\frac{\pi}{2}$, since this is guaranteed with $N\geq2$.
Continuing \eqref{eq.gamma} gives
\begin{equation}
\label{eq.cap hypersurface area} 
\mathrm{H}^{M-1}(C(\gamma))
\geq\tfrac{2\pi^{(M-1)/2}}{\Gamma(\frac{M-1}{2})}\int_0^\gamma\big(\tfrac{2\phi}{\pi}\big)^{M-2}\mathrm{d}\phi
=\tfrac{(2\gamma)^{M-1}}{(M-1)\pi^{(M-3)/2}\Gamma(\frac{M-1}{2})}.
\end{equation}
Using the formula for a hypersphere's hypersurface area, we can express the left-hand side of \eqref{eq.cap hypersurface area}:
\begin{equation*}
\tfrac{(2\gamma)^{M-1}}{(M-1)\pi^{(M-3)/2}\Gamma(\frac{M-1}{2})}
\leq \mathrm{H}^{M-1}(C(\gamma))
=\tfrac{1}{2N}\mathrm{H}^{M-1}(S^{M-1})
=\tfrac{\pi^{M/2}}{N\Gamma(\frac{d}{2})}.
\end{equation*}
Isolating $2\gamma$ above and using $\alpha\leq2\gamma$ and $\mu=\cos\alpha$ gives \eqref{eq.bound}.
The second part of the result comes from a simple application of Stirling's approximation.
\end{proof}

In \cite{CHS96}, numerical results are given for $M=3$, and we compare these results to Theorems~\ref{thm.complex bound} and~\ref{thm.bound} in Figure~\ref{figure}.
Considering this figure, we note that the bound in Theorem~\ref{thm.complex bound} is inferior to the maximum of the Welch bound and the bound in Theorem~\ref{thm.bound}, at least when $M=3$.
This illustrates the degree to which Theorem~\ref{thm.bound} improves the bound in Theorem~\ref{thm.complex bound} for real frames.
In fact, since $\cos(\frac{\pi}{a})\geq 1-\frac{2}{a}$ for all $a\geq2$, the bound for real frames in Theorem~\ref{thm.bound} is asymptotically better than the bound for complex frames in Theorem~\ref{thm.complex bound}.
Moreover, for $M=2$, Theorem~\ref{thm.bound} says $\mu\geq\cos(\frac{\pi}{N})$, and \cite{BK06} proved this bound to be tight for every $N\geq2$.
Lastly, Figure~\ref{figure} illustrates that Theorem~\ref{thm.3d points} improves the bound in Theorem~\ref{thm.bound} for the case $M=3$.

In many applications, large dictionaries are built to obtain sparse reconstruction, but the known guarantees on sparse reconstruction place certain requirements on worst-case coherence.
Asymptotically, the bounds in Theorems~\ref{thm.complex bound} and \ref{thm.bound} indicate that certain exponentially large dictionaries will not satisfy these requirements.
For example, if $N=\Theta(3^M)$, then $\mu_F=\Omega(\frac{1}{3})$ by Theorem~\ref{thm.complex bound}, and if the frame is real, we have $\mu_F=\Omega(\frac{1}{2})$ by Theorem~\ref{thm.bound}.
Such a dictionary will only work for sparse reconstruction if the sparsity level $K$ is sufficiently small; deterministic guarantees require $K<\mu_F^{-1}$ \cite{donoho:tit06b,tropp:tit06}, while probabilistic guarantees require $K<\mu_F^{-2}$ \cite{bajwa:jcn10,tropp:cras08}, and so in this example, the dictionary can, at best, only accommodate sparsity levels that are smaller than 10.
Unfortunately, in real-world applications, we can expect the sparsity level to scale with the signal dimension.
This in mind, Theorems~\ref{thm.complex bound} and \ref{thm.bound} tell us that dictionaries can only be used for sparse reconstruction if $N=O((2+\epsilon)^M)$ for some sufficiently small $\epsilon>0$.  
To summarize, the Welch bound is known to be tight only if $N\leq M^2$, and Theorems~\ref{thm.complex bound} and \ref{thm.bound} give bounds which are asympotically better than the Welch bound whenever $N=\Omega(2^M)$.
When $N$ is between $M^2$ and $2^M$, the best bound to date is the (loose) Welch bound, and so more work needs to be done to bound worst-case coherence in this parameter region.

\section{Reducing average coherence}

In \cite{bajwa:jcn10}, average coherence is used to derive a number of guarantees on sparse signal processing.
Since average coherence is so new to the frame theory literature, this section will investigate how average coherence relates to worst-case coherence and the spectral norm.
We start with a definition:

\begin{definition}[Wiggling and flipping equivalent frames]
\label{def.flipping and wiggling}
We say the frames $F$ and $G$ are \emph{wiggling equivalent} if there exists a diagonal matrix $D$ of unimodular entries such that $G=FD$.
Furthermore, they are \emph{flipping equivalent} if $D$ is real, having only $\pm1$'s on the diagonal.
\end{definition}

The terms ``wiggling'' and ``flipping'' are inspired by the fact that individual frame elements of such equivalent frames are related by simple unitary operations.
Note that every frame with $N$ nonzero frame elements belongs to a flipping equivalence class of size $2^N$, while being wiggling equivalent to uncountably many frames.
The importance of this type of frame equivalence is, in part, due to the following lemma, which characterizes the shared geometry of wiggling equivalent frames:

\begin{lemma}[Geometry of wiggling equivalent frames]\label{lem:geom_eqframes}
Wiggling equivalence preserves the norms of frame elements, the worst-case coherence, and the spectral norm.
\end{lemma}

\begin{proof}
Take two frames $F$ and $G$ such that $G=FD$.
The first claim is immediate.
Next, the Gram matrices are related by $G^*G=D^*F^*FD$.
Since corresponding off-diagonal entries are equal in modulus, we know the worst-case coherences are equal.
Finally, $\|G\|_2^2=\|GG^*\|_2^2=\|FDD^*F^*\|_2^{}=\|FF^*\|_2^{}=\|F\|_2^2$, and so we are done.
\end{proof}

Wiggling and flipping equivalence are not entirely new to frame theory.
For a real equiangular tight frame $F$, the Gram matrix $F^*F$ is completely determined by the sign pattern of the off-diagonal entries, which can in turn be interpreted as the Seidel adjacency matrix of a graph $G_F$.
As such, flipping a frame element $f\in F$ has the effect of negating the corresponding row and column in the Gram matrix, which further corresponds to \emph{switching} the adjacency rule for that vertex $v_f\in V(G_F)$ in the graph---vertices are adjacent to $v_f$ after switching precisely when they were not adjacent before switching. 
Graphs are called \emph{switching equivalent} if there is a sequence of switching operations that produces one graph from the other; this equivalence was introduced in \cite{vanlint:im66} and was later extensively studied by Seidel in \cite{S73,seidel:laa68}.
Since flipping equivalent real equiangular tight frames correspond to switching equivalent graphs, the terms have become interchangeable.
For example, \cite{bodmann:jfa10} uses switching (i.e., wiggling and flipping) equivalence to make progress on an important problem in frame theory called the \emph{Paulsen problem}, which asks how close a nearly unit norm, nearly tight frame must be to a unit norm tight frame.

Now that we understand wiggling and flipping equivalence, we are ready for the main idea behind this section.
Suppose we are given a unit norm frame with acceptable spectral norm and worst-case coherence, but we also want the average coherence to satisfy (SCP-2).
Then by Lemma~\ref{lem:geom_eqframes}, all of the wiggling equivalent frames will also have acceptable spectral norm and worst-case coherence, and so it is reasonable to check these frames for good average coherence.
In fact, the following theorem guarantees that at least one of the flipping equivalent frames will have good average coherence, with only modest requirements on the original frame's redundancy.

\begin{theorem}[Constructing frames with low average coherence]\label{thm:avc_rand}
Let $F$ be an $M\times N$ unit norm frame with $\smash{M < \frac{N-1}{4\log 4N}}$.
Then there exists a frame $G$ that is flipping equivalent to $F$ and satisfies $\smash{\nu_G\leq\frac{\mu_G}{\sqrt{M}}}$.
\end{theorem}

\begin{proof}
Take $\{R_n\}_{n=1}^N$ to be a Rademacher sequence that independently takes values $\pm1$, each with probability $\frac{1}{2}$.
We use this sequence to randomly flip $F$; define $Z:=F~\mathrm{diag}\{R_n\}_{n=1}^N$.
Note that if $\smash{\Pr(\nu_Z\leq\frac{\mu_F}{\sqrt{M}})>0}$, we are done.
Fix some $i\in\{1,\ldots,N\}$.
Then
\begin{equation}
\label{pfeqn:avc_rand_tail1}
\Pr\Bigg(\tfrac{1}{N-1} \bigg|\sum_{\substack{j=1\\ j\neq i}}^N \langle z_i,z_j\rangle\bigg| > \tfrac{\mu_F}{\sqrt{M}}\Bigg) 
=\Pr\Bigg(\bigg|\sum_{\substack{j=1\\ j\neq i}}^N R_j\langle f_i,f_j\rangle\bigg| > \tfrac{(N-1)\mu_F}{\sqrt{M}}\Bigg). 
\end{equation}
We can view $\sum_{j\neq i} R_j\langle f_i,f_j\rangle$ as a sum of $N-1$ independent zero-mean complex random variables that are bounded by $\mu_F$.
We can therefore use a complex version of Hoeffding's inequality \cite{hoeffding:jasa63} (see, e.g., \cite[Lemma~3.8]{bajwa:thesis}) to bound the probability expression in \eqref{pfeqn:avc_rand_tail1} as $\leq4\mathrm{e}^{-(N-1)/4M}$.
From here, a union bound over all $N$ choices for $i$ gives $\Pr(\nu_Z\leq\frac{\mu_F}{\sqrt{M}})\geq 1-4N\mathrm{e}^{-(N-1)/4M}$,
and so $M < \frac{N-1}{4\log 4N}$ implies $\Pr(\nu_Z\leq\frac{\mu_F}{\sqrt{M}})>0$, as desired.
\end{proof}

While Theorem~\ref{thm:avc_rand} guarantees the existence of a flipping equivalent frame with good average coherence, the result does not describe how to find it.
Certainly, one could check all $2^N$ frames in the flipping equivalence class, but such a procedure is computationally slow.
As an alternative, we propose a linear-time flipping algorithm (Algorithm~\ref{alg:flipping}).
The following theorem guarantees that linear-time flipping will produce a frame with good average coherence, but it requires the original frame's redundancy to be higher than what suffices in Theorem~\ref{thm:avc_rand}.

\begin{algorithm*}[t]
\caption{Linear-time flipping}
\label{alg:flipping}
\textbf{Input:} An $M\times N$ unit norm frame $F$\\
\textbf{Output:} An $M\times N$ unit norm frame $G$ that is flipping equivalent to $F$
\begin{algorithmic}
\STATE $g_1\leftarrow f_1$ \hfill \COMMENT{Keep first frame element}
\FOR{$n=2$ to $N$}
\IF{$\|\sum_{i=1}^{n-1}g_i+f_n\|\leq\|\sum_{i=1}^{n-1}g_i-f_n\|$} 
\STATE $g_n\leftarrow f_n$ \hfill \COMMENT{Keep frame element to make sum length shorter}
\ELSE
\STATE $g_n\leftarrow -f_n$ \hfill \COMMENT{Flip frame element to make sum length shorter}
\ENDIF
\ENDFOR
\end{algorithmic}
\end{algorithm*}

\begin{theorem}
\label{thm.alg}
Suppose $N\geq M^2+3M+3$.
Then Algorithm~\ref{alg:flipping} outputs an $M\times N$ frame $G$ that is flipping equivalent to $F$ and satisfies $\nu_G\leq\frac{\mu_G}{\sqrt{M}}$.
\end{theorem}

\begin{proof}
Considering Lemma~\ref{lem.sufficient conditions}(iii), it suffices to have $\|\sum_{n=1}^N g_n\|^2\leq N$.
We will use induction to show $\|\sum_{n=1}^k g_n\|^2\leq k$ for $k=1,\ldots,N$.
Clearly, $\|\sum_{n=1}^1 g_n\|^2=\|f_n\|^2=1\leq1$.
Now assume $\|\sum_{n=1}^k g_n\|^2\leq k$.
Then by our choice for $g_{k+1}$ in Algorithm~\ref{alg:flipping}, we know that
$\|\sum_{n=1}^kg_n+g_{k+1}\|^2\leq\|\sum_{n=1}^kg_n-g_{k+1}\|^2$.
Expanding both sides of this inequality gives
\begin{equation*}
\bigg\|\sum_{n=1}^kg_n\bigg\|^2+2\mathrm{Re}\bigg\langle\sum_{n=1}^kg_n,g_{k+1}\bigg\rangle+\|g_{k+1}\|^2
\leq\bigg\|\sum_{n=1}^kg_n\bigg\|^2-2\mathrm{Re}\bigg\langle\sum_{n=1}^kg_n,g_{k+1}\bigg\rangle+\|g_{k+1}\|^2,
\end{equation*}
and so $\mathrm{Re}\langle\sum_{n=1}^kg_n,g_{k+1}\rangle\leq0$.
Therefore,
\begin{equation*} 
\bigg\|\sum_{n=1}^{k+1}g_n\bigg\|^2
=\bigg\|\sum_{n=1}^kg_n\bigg\|^2+2\mathrm{Re}\bigg\langle\sum_{n=1}^kg_n,g_{k+1}\bigg\rangle+\|g_{k+1}\|^2
\leq\bigg\|\sum_{n=1}^kg_n\bigg\|^2+\|g_{k+1}\|^2
\leq k+1,
\end{equation*}
where the last inequality uses the inductive hypothesis.
\end{proof}

\begin{example}
As an example of how linear-time flipping reduces average coherence, consider the following matrix:
\begin{equation*}
F:= \frac{1}{\sqrt{5}}\left[ \begin{array}{cccccccccc} +&+&+&+&-&+&+&+&+&-\\+&-&+&+&+&-&-&-&+&-\\+&+&+&+&+&+&+&+&-&+\\-&-&-&+&-&+&+&-&-&-\\-&+&+&-&-&+&-&-&-&- \end{array} \right].
\end{equation*}
Here, $\smash{\nu_F\approx0.3778>0.2683\approx\frac{\mu_F}{\sqrt{M}}}$.
Even though $N<M^2+3M+3$, we run linear-time flipping to get the flipping pattern $D:=\mathrm{diag}(+-+--++-++)$.
Then $FD$ has average coherence $\smash{\nu_{FD}\approx0.1556<\frac{\mu_{F}}{\sqrt{M}}=\frac{\mu_{FD}}{\sqrt{M}}}$.
This example illustrates that the condition $N\geq M^2+3M+3$ in Theorem~\ref{thm.alg} is sufficient but not necessary.
\end{example}

\section*{Acknowledgments}
The authors thank the anonymous referees for their helpful suggestions, Matthew Fickus for his insightful comments on chirp frames, and Samuel Feng and Michael A.~Schwemmer for their help with using the computer clusters in Princeton's mathematics department.
This work was supported by the Office of Naval Research under grant N00014-08-1-1110,
by the Air Force Office of Scientific Research under grants FA9550-09-1-0551 and
FA 9550-09-1-0643, and by NSF under grant DMS-0914892.
Mixon was supported by the A.B. Krongard Fellowship.
The views expressed in this article are those of the authors and do not reflect the official policy or position of the United States Air Force, Department of Defense, or the U.S. Government.

\end{document}